\documentclass[11pt,a4paper,leqno]{article}
\usepackage{amsmath}
\usepackage{amsfonts}
\usepackage{amssymb}
\usepackage{amsthm}
\usepackage{hyperref}
\usepackage[margin=2.5cm]{geometry}

\newtheorem{theorem}{Theorem}[section]
\newtheorem{corollary}[theorem]{Corollary}
\newtheorem{lemma}[theorem]{Lemma}

\numberwithin{equation}{section}

\title{Convergence of linear combinations of iterates of an inner function}

\author{Artur Nicolau\thanks{Supported in part by the Generalitat de Catalunya (grant 2017 SGR 395) and the Spanish Ministerio de Ciencia e Innovaci\'on (project  MTM2017-85666-P)}\\
{\small
\begin{tabular}{@{}c}
Departament de Matem\`atiques\\
Universitat Aut\`onoma de Barcelona\\
and Centre de Recerca Matem\`atica\\
08193 Barcelona\\
{\tt artur@mat.uab.cat}
\end{tabular}}}

\date{}
        
\begin{document}
\maketitle
\begin{abstract}
\noindent Let $f$ be an inner function with $f(0)=0$ which is not a rotation and let $f^{n}$ be its $n$-th iterate. Let $\{a_{n}\}$ be a sequence of complex numbers. We prove that the series  $\sum a_{n}f^{n}(\xi)$ converges at almost every point~$\xi$ of the unit circle if and only if $\sum |a_n|^2 < \infty$. The main step in the proof is to show that under this assumption, the function $F= \sum a_n f^n$ has bounded mean oscillation. We also prove that $F$ is bounded on the unit disc if and only if $\sum |a_n| < \infty$. Finally we describe the sequences of coefficients $\{a_n \}$ such that $F$ belongs to other classical function spaces, as the disc algebra and the Dirichlet class. 
\end{abstract}
\section{Introduction}

 Inner functions are analytic mappings from the unit disc~$\mathbb{D}$ of the complex plane into itself whose radial limits are of modulus one at almost every point of the unit circle~$\partial\mathbb{D}$. Inner functions are a central notion in Complex and Functional Analysis. See for instance \cite{Ga}. Any inner function~$f$ induces a map from~$\partial\mathbb{D}$ into itself defined at almost every point~$\xi\in \partial\mathbb{D}$ by $f(\xi)=\lim\limits_{r\to 1}f(r\xi)$. Let $f^{n}$ denote the $n$-th iterate of~$f$,  which is again an inner function. The  main purpose of this paper is to study the convergence of linear combinations of iterates of an inner function and to show that in a certain sense, the iterates~$f^{n}\colon \partial\mathbb{D}\to \partial\mathbb{D}$ behave as independent random variables. This phenomena has been extensively studied in the context of random series of functions (\cite{Ka}) but it also occurs in other settings where independence is not present, as in the theory of lacunary series (see Section 6 of Chapter V and Section 5 of Chapter XVI of \cite{Zy}).

We start discussing pointwise convergence. Let $\{X_{n}\}$ be a sequence of independent random variables with mean~$0$ and finite variances~$V(X_{n})$. The classical Khintchine--Kolmogorov Convergence Theorem asserts that $\sum X_{n}$ converges almost surely if and only if $\sum V(X_{n})<\infty$. In our context we have the following analogous result.

\begin{theorem}\label{theo1.1}
Let $f$ be an inner function with $f(0)=0$ which is not a rotation and let $\{a_{n}\}$ be a sequence of complex numbers. Then the following conditions are equivalent:
\begin{enumerate}
\item[(a)] The series $\sum\limits_{n=1}^{\infty} a_{n}f^{n}(\xi)$ converges at almost every point~$\xi\in \partial\mathbb{D}$.

\item[(b)] The set $\Bigl\{\xi\in \partial\mathbb{D}:\sup\limits_{N}\Bigl|\sum\limits_{n=1}^{N} a_{n}f^{n}(\xi)\Bigr|<\infty\Bigr\}$ has positive Lebesgue measure.

\item[(c)] The complex numbers $\{a_{n}\}$ satisfy $\sum\limits_{n=1}^{\infty} |a_{n}|^{2}<\infty$.
\end{enumerate}
\end{theorem}

Let $m$ denote normalized Lebesgue measure in $\partial\mathbb{D}$. For $0<p<\infty$ let $\mathbb{H}^{p}(\mathbb{D})$ be the classical Hardy space of analytic functions~$F$ in~$\mathbb{D}$ such that
$$
\|F\|_{p}^{p}=\sup_{r<1}\int_{\partial\mathbb{D}}|F(r\xi)|^{p}\,dm(\xi)<\infty.
$$
Let $\mathbb{H}^{\infty}(\mathbb{D})$ be the algebra of bounded analytic functions~$F$ in~$\mathbb{D}$ and $\|F\|_{\infty}=\sup \{|F(z)|: z\in \mathbb{D}\}$. Let~$\operatorname{BMOA}$ be the space of analytic functions~$F$ in~$\mathbb{D}$ such that
$$
\|F\|^{2}_{\operatorname{BMOA}(\mathbb{D})}=\sup_{z\in \mathbb{D}} \int_{\partial\mathbb{D}} |F(\xi)-F(z)|^{2} P(z,\xi)\,dm(\xi)<\infty.
$$
Here $P(z,\xi)=(1-|z|^{2})|\xi-z|^{-2}$ is the Poisson kernel. The subspace~$\operatorname{VMOA}$ is formed by the functions~$F\in\operatorname{BMOA}$ such that the integral above tends to $0$ as $|z|$ tends to~$1$.
Any function~$F\in \mathbb{H}^{p}(\mathbb{D})$, $0 <  p\le\infty$, has non-tangential limit, denoted by~$F(\xi)$, at almost every point~$\xi\in \partial\mathbb{D}$. Moreover if  $p\ge 1$, then 
$$
F(z)=\int_{\partial\mathbb{D}}F(\xi)P(z,\xi)\,dm(\xi),\quad z\in \mathbb{D}.
$$
We have $\mathbb{H}^{\infty}(\mathbb{D})\subseteq \operatorname{BMOA}(\mathbb{D})\subseteq \mathbb{H}^{p}(\mathbb{D})$ for any $0<p<\infty$. Actually $\operatorname{BMOA}(\mathbb{D})$ is the natural substitute of~$\mathbb{H}^{\infty}(\mathbb{D})$ in several results of the theory of Hardy spaces. See \cite{Ga} for all these well known results. 

The proof of Theorem~\ref{theo1.1} is based in the following $\operatorname{BMOA}$ type estimate.

\begin{theorem}\label{theo1.2}
Let $f$ be an inner function with $f(0)=0$ which is not a rotation. Then there exists a constant $C=C(f)>0$ such that for any positive integer~$N$, any set $\{a_{n}:n=1,\dots,N\}$ of complex numbers and any point $z\in\mathbb{D}$, we have
$$
\int_{\partial\mathbb{D}}\left| \sum_{n=1}^{N} a_{n}(f^{n}(\xi)-f^{n}(z))\right|^{2} P(z,\xi)\,dm(\xi)\le C\sum^{N}_{n=1}|a_{n}|^{2} (1-|f^{n}(z)|^{2}).
$$
\end{theorem}

Let $f$ be an analytic selfmapping of $\mathbb{D}$ with $f(0)=0$ which is not a rotation. The classical Denjoy--Wolff Theorem asserts that $f^{n}$ tends to~$0$ uniformly on compact sets of~$\mathbb{D}$. See Chapter V of \cite{Sh}. The proof of Theorem~\ref{theo1.2} uses the interplay between those dynamical properties of the inner function~$f$ as a self mapping of $\partial\mathbb{D}$ and those as a self mapping of~$\mathbb{D}$. In particular we use a result due to Pommerenke (\cite{Po}) which provides an exponential decay in the Denjoy--Wolff Theorem. 

It was proved in~\cite{NS2} that $\sum a_{n}f^{n}$ converges in $\mathbb{H}^{2}(\mathbb{D})$ if and only if $\sum |a_{n}|^{2}<\infty$. As a consequence of Theorem~\ref{theo1.2} we will show that under this last condition, we actually have $\sum a_{n}f^{n}\in\operatorname{BMOA}(\mathbb{D})$.

\begin{corollary}\label{coro1.3}
Let $f$ be an inner function with $f(0)=0$ which is not a rotation and let $\{a_{n}\}$ be a sequence of complex numbers. Assume $\sum\limits^{\infty}_{n=1}|a_{n}|^{2}<\infty$.

\begin{enumerate}
\item[(a)] Then $F=\sum\limits_{n=1}^{\infty} a_{n}f^{n}\in \operatorname{BMOA}(\mathbb{D})$ and for almost every $\xi\in\partial\mathbb{D}$ we have
\begin{equation}\label{eq1.1}
\lim_{r\to 1}F(r\xi)=\sum^{\infty}_{n=1}a_{n}f^{n}(\xi).
\end{equation}
Moreover there exists a constant~$C=C(f)>0$ only depending on~$f$, such that
$$
C^{-1} \sum^{\infty}_{n=1}|a_{n}|^{2} \le \|F\|^{2}_{\operatorname{BMOA}(\mathbb{D})}\le C\sum^{\infty}_{n=1}|a_{n}|^{2}.
$$

\item[(b)] If $f$ is a finite Blaschke product, then $F=\sum\limits^{\infty}_{n=1}a_{n}f^{n}\in\operatorname{VMOA}$.
\end{enumerate}
\end{corollary}

The following Khintchine type estimate follows easily from the previous result.

\begin{corollary}\label{coro1.4}
Let $f$ be an inner function with $f(0)=0$ which is not a rotation and let $\{a_{n}\}$ be a sequence of complex numbers. Assume $\sum\limits^{\infty}_{n=1}|a_{n}|^{2}<\infty$. Then $\sum\limits^{\infty}_{n=1}a_{n} f^{n}\in \mathbb{H}^{p}(\mathbb{D})$ for any~$0<p<\infty$. Moreover for any $0<p<\infty$, there exists a constant~$C(p,f)>0$ such that for any sequence of complex numbers~$\{a_{n}\}$ we have
$$
C(p,f)^{-1}\left(	\sum^{\infty}_{n=1}|a_{n}|^{2}\right)^{1/2} \le\left\|\sum_{n=1}^{\infty}a_{n}f^{n}\right\|_{p}\le C(p,f)\left(\sum^{\infty}_{n=1}|a_{n}|^{2}\right)^{1/2}.
$$
\end{corollary}

By Theorem~\ref{theo1.1}, if $\sum |a_{n}|^{2}<\infty$, the series~$\sum a_{n}f^{n}(\xi)$ converges at almost every $\xi\in\partial\mathbb{D}$, while if $\sum |a_{n}|^{2}=\infty$, the series~$\sum a_{n}f^{n}(\xi)$ converges at almost no point~$\xi\in\partial\mathbb{D}$.~If the coefficients satisfy the stronger condition~$\sum |a_{n}|<\infty$, it is clear that $\sum a_{n}f^{n}\in\mathbb{H}^{\infty}(\mathbb{D})$. Our next result says that this condition is also necessary.

\begin{theorem}\label{theo1.5}
Let $f$ be an inner function with $f(0)=0$ which is not a rotation. Let $\{a_{n}\}$ be a sequence of complex numbers with $\sum\limits_{n=1}^\infty |a_{n}|^{2}<\infty$ and $F=\sum\limits_{n=1}^\infty a_{n}f^{n}$. Assume there exists an arc~$I\subset \partial\mathbb{D}$ such that $\sup\{|F(\xi)|:\xi\in I\}<\infty$. Then $\sum\limits_{n=1}^\infty |a_{n}|<\infty$.
\end{theorem}

The proof of Theorem~\ref{theo1.5} is the most technical part of the paper.~The main idea is to construct by induction a sequence of arcs~$I_{k}\subset I$ and a sequence~$M_{k}$ of positive integers tending to infinity such that
\begin{equation}\label{eq1.2}
\frac{1}{m(I_{k})} \int_{I_{k}}\operatorname{Re}F\,dm\ge A\sum^{M_{k}}_{j=1}|a_{j}|-B,\quad k=1,2,\dotsc,
\end{equation}
where $A$ and $B$ are positive constants independent of~$k$.~We start showing that if $|f(z)|$ is sufficiently small, the two terms in the conclusion of Theorem~\ref{theo1.2} are comparable. This provides a way of finding arcs~$I_{k}\subset I$ and of splitting $F=\sum\limits_{j}F_{j}$ into blocks~$F_{k}$ such that
$$
\frac{1}{m(I_{k})} \int_{I_{k}}\sum_{j=1}^{k}
\operatorname{Re}F_{j}\,dm\ge C\sum^{k}_{j=1}\|\operatorname{Re}F_{j}\|_{2},
\quad k=1,2,\dotsc
$$
Since the blocks~$F_{j}$ have a uniformly bounded number of terms, $\|\operatorname{Re}F_{j}\|_{2}$ can be estimated from below by the sum of the modulus of the corresponding coefficients and \eqref{eq1.2} follows easily.

Let $\overline{\mathbb{H}^{\infty}(\mathbb{D})}$ denote the closure of $\mathbb{H}^{\infty}(\mathbb{D})$ in $\operatorname{BMOA}(\mathbb{D})$. This space was studied in~\cite{GJ}. See also~\cite{NS1} and~\cite{SS}. Theorems~\ref{theo1.1}, \ref{theo1.5} and Corollary~\ref{coro1.3} lead to the following result.

\begin{corollary}\label{coro1.6}
Let $f$ be an inner function with $f(0)=0$ which is not a rotation. Let $\{a_{n}\}$ be a sequence of complex numbers with $\sum\limits_{n=1}^\infty |a_{n}|^{2}<\infty$ but $\sum\limits_{n=1}^\infty |a_{n}|=\infty$. Then $F=\sum\limits^{\infty}_{n=1}a_{n}f^{n}\in \overline{\mathbb{H}^{\infty}(\mathbb{D})}$. Moreover 
$$
F(\xi)=\lim_{r\to 1}F(r\xi)=\sum^{\infty}_{n=1}a_{n}f^{n}(\xi)
$$
at almost every $\xi\in \partial\mathbb{D}$ but $\sup\{|F(\xi)|:\xi\in I\}=\infty$ for any arc~$I\subset \partial\mathbb{D}$.
\end{corollary}

Theorems~\ref{theo1.1} and \ref{theo1.5} and Corollaries~\ref{coro1.3} and \ref{coro1.4} can be understood as analogues to classical results in the theory of lacunary series (see Section 6 of Chapters V and VI of~\cite{Zy}). However it should be noted that no lacunarity assumption is needed in our results.

The paper is organized as follows. Next section is devoted to the proof of Theorems~\ref{theo1.1}, \ref{theo1.2} and Corollaries~\ref{coro1.3} and \ref{coro1.4}. In Section~\ref{sec3} we start collecting several auxiliary results which are used in the proof of Theorem~\ref{theo1.5}. In Section~\ref{sec4} we describe in terms of the coefficients, those linear combinations of iterates which belong to certain classical function spaces, such as the disc algebra, the Dirichlet space or the Bloch space. Finally last section is devoted to state some related open questions we have not explored. 

It is a pleasure to thank Od\'i Soler i Gibert for some very helpful comments. 

\section{BMO estimates and pointwise convergence}\label{sec2}

Let $\rho(z,w)$ denote the pseudohyperbolic distance between the points~$z,w\in\mathbb{D}$ defined as $\rho(z,w)=|z-w||1-\overline wz|^{-1}$. Schwarz's Lemma asserts that any analytic selfmapping~$f$ of~$\mathbb{D}$ contracts hyperbolic distances, that is, $\rho(f(z),f(w))\le \rho(z,w)$ for any $z,w\in\mathbb{D}$. Equivalently ${D}_{h}(f)(z)\le 1$, $z\in \mathbb{D}$, where
\begin{equation}\label{eq2.1}
D_{h}(f)(z)=\frac{(1-|z|^{2})|f'(z)|}{1-|f(z)|^{2}},\quad z\in \mathbb{D},
\end{equation}
is the hyperbolic derivative of~$f$ at the point~$z$. Our first auxiliary result is a quantitative version of the Denjoy--Wolff Theorem which is essentially due to Pommerenke (\cite{Po}). Its short proof is included for the sake of completeness.

\begin{lemma}\label{lem2.1}
Let $f\in H^{\infty}(\mathbb{D})$, $\|f\|_{\infty}\le 1$ with $f(0)=0$ which is not a rotation.
\begin{enumerate}
\item[(a)] Then, for any $0<r<1$, there exists $0<c=c(r,f)<1$ such that $1-|z|\le c(1-|f(z)|)$ if $|z|\ge r$.

\item[(b)] Assume $f'(0) \neq 0$. Then there exists $0<r_{0}=r_{0}(f)<1$ such that
$$
|f^{n}(z)|\le r_{0}^{-1} |f'(0)|^{n}|z|,\quad n\ge 1,\quad\text{if}\quad |z|\le r_{0}.
$$

\item[(c)] Assume $f'(0)=0$. Then $|f^{n}(z)|\le |z|^{2^{n}}$, $z\in \mathbb{D}$.
\end{enumerate}
\end{lemma}

\begin{proof}
Since $f(0)=0$, Schwarz's Lemma gives
$$
\rho\left(\frac{f(z)}{z},f'(0)\right)\le |z|,\quad z\in \mathbb{D}.
$$
Denote $\psi(z)=z(z+|f'(0)|)(1+|f'(0)|z)^{-1}$, $z\in \mathbb{D}$, to obtain
\begin{equation}\label{eq2.2}
|f(z)|\le \psi (|z|),\quad z\in \mathbb{D}.
\end{equation}
Then 
$$
1-|f(z)|\ge 1-\psi (|z|)=(1-|z|)\frac{1+|z|}{1+|f'(0)||z|}.
$$
Given $0<r<1$, taking  $c=(1+r)^{-1}(1+|f'(0)|r)$ the estimate in (a) follows.

We now prove (b). For $n=1,2,\dotsc$, consider the function~$g_{n}(z)=\psi^{n}(z)|f'(0)|^{-n}$. It is known that $\{g_{n}\}$ converges uniformly on compact sets of~$\mathbb{D}$ to an analytic function~$g$ on~$\mathbb{D}$, known as the K\"onigs function of~$\psi$, which satisfies $g(\psi(z))=|f'(0)|g(z)$, $z\in \mathbb{D}$. See \cite[pp.~89--93]{Sh}. Moreover for $0\le x\le 1$ we have
$$
g_{n+1}(x)=\frac{\psi^{n+1}(x)}{|f'(0)|^{n+1}} =g_{n}(x) \frac{1+|f'(0)|^{n-1}g_{n}(x)} {1+|f'(0)|^{n+1}g_{n}(x)}\ge g_{n}(x).
$$
Hence $g_{n}(x)\le g(x)$, $n=1,2,\dotsc,0<x<1$. Note that there exists a constant $\delta=\delta (|f'(0)|)>0$ such that $\psi$ is univalent in $\{z\in \mathbb{D}:|z|<\delta\}$. Hence $g_{n}$ is also univalent in $\{z\in \mathbb{D}: |z|<\delta\}$. By Koebe Distortion Theorem, there exists $0<r_{0}=r_{0}(f)<1$ such that $|g(w)|<1$ if $|w|<r_{0}$. By Schwarz's Lemma we deduce $|g(w)| \leq r_{0}^{-1}|w|$ if $|w|<r_{0}$. Applying \eqref{eq2.2}, for any $|z| < r_{0}$ we have
$$
|f^{n}(z)|\le \psi^{n} (|z|)\le |f'(0)|^{n} g(|z|)\le r_{0}^{-1} |f'(0)|^{n}|z|,\quad n=1,2,\dotsc
$$
This proves (b). Assume now $f'(0)=0$. Note that \eqref{eq2.2} gives $|f(z)|\le |z|^{2}$, $z\in \mathbb{D}$. Iterating we obtain (c).
\end{proof}

For future reference we state the following easy consequence of Lemma~\ref{lem2.1}.

\begin{corollary}\label{coro2.2}
Let $f\in \mathbb{H}^{\infty}(\mathbb{D})$, $\|f\|_{\infty}\le 1$ with $f(0)=0$ which is not a rotation. Then there exist constants $0<r_{0}=r_{0}(f)<1$ and $0<c_{0}=c_{0}(f)<1$ such that
$$
|f^{n}(z)|\le r_{0}^{-1} c_{0}^{n}|z|\quad \text{if}\quad |z|\le r_{0}.
$$
\end{corollary}

\begin{proof}
If $f'(0)\ne 0$ let $r_{0}=r_{0}(f)$ be the constant given by part~(b) of Lemma~\ref{lem2.1} and $c_{0}=|f'(0)|$. If $f'(0)=0$, pick any $0<r_{0}<1$. Part~(c) of Lemma~\ref{lem2.1} gives $|f^{n}(z)|\le r_{0}^{2^{n}-1}|z|$, $|z|\le r_{0}$. Since $2^{n}-1\ge n$, $n\ge 1$, we can take $c_{0}=r_{0}$.
\end{proof}

For future reference we also state the following auxiliary result whose proof can be found in \cite[Theorem~9]{NS2}.

\begin{lemma}\label{lem2.3}
Let $f$ be an inner function with $f(0)=0$ which is not a rotation. Let $\{a_{n}\}$ be a sequence of complex numbers. Then
$$
\frac{1-|f'(0)|}{1+|f'(0)|} \sum^{N}_{n=1}|a_{n}|^{2} \le  \left\|\sum^{N}_{n=1}a_{n}f^{n}\right\|_2^2 \le \frac{1+|f'(0)|}{1-|f'(0)|} \sum^{N}_{n=1}|a_{n}|^{2},\quad N=1,2,\dotsc
$$
\end{lemma}

We are now ready to prove Theorem~\ref{theo1.2}.

\begin{proof}[Proof of Theorem~\ref{theo1.2}]
Without loss of generality we can assume $f^{n}(z)\ne 0$ for any $n\ge 1$. Denote $F=\sum\limits^{N}_{n=1}a_{n}f^{n}$. Then
$$
\int_{\partial \mathbb{D}}|F(\xi)-F(z)|^{2}P(z,\xi)\,dm (\xi)=\sum^{N}_{n=1}|a_{n}|^{2}c_{n,n}+2\operatorname{Re}\sum^{N-1}_{n=1}\sum^{N}_{k>n}\overline{a}_{n}a_{k}c_{k,n},
$$
where
$$
c_{k,n}=\int_{\partial \mathbb{D}}\overline{(f^{n}(\xi)-f^{n}(z))} (f^{k}(\xi)-f^{k}(z))P(z,\xi)\,dm(\xi),\quad k,n=1,\dotsc,N.
$$
For $k\ge n$, the function~$\overline{f^{n}(\xi)} f^{k}(\xi)=f^{k}(\xi)/f^{n}(\xi)$, $\xi\in\partial \mathbb{D}$, has an analytic extension to~$\mathbb{D}$. Then
$$
c_{k,n}=\frac{f^{k}(z)}{f^{n}(z)}-f^{k}(z) \overline{f^{n}(z)}=\frac{f^{k}(z)(1-|f^{n}(z)|^{2})}{f^{n}(z)},\quad k\ge n.
$$
Hence
\begin{equation}\label{eq2.3}
\int_{\partial \mathbb{D}}|F(\xi)-F(z)|^{2}
P(z,\xi)=A+2\operatorname{Re}B,
\end{equation}
where
\begin{align*}
A&=\sum^{N}_{n=1} |a_{n}|^{2} (1-|f^{n}(z)|^{2}),\\*[5pt]
B&=\sum^{N-1}_{n=1}\sum^{N}_{k>n}\overline{a}_{n} a_{k}\frac{f^{k}(z)(1-|f^{n}(z)|^{2})}{f^{n}(z)}.
\end{align*}
The idea is that the cross term~$B$ can be estimated by the diagonal term~$A$. Let $0<r_{0}=r_{0}(f)<1$ and $0<c_{0}=c_{0}(f)<1$ be the constants appearing in Corollary~\ref{coro2.2}. The Denjoy-Wolff Theorem says that the iterates $f^n$ converge to $0$ uniformly on compacts of $\mathbb{D}$. So we can consider the smallest positive integer $\ell=\ell(z)$ such that $|f^{\ell}(z)|\le r_{0}$. Corollary~\ref{coro2.2} gives that $|f^{k}(z)|\le r_{0}^{-1} c_{0}^{k-n} |f^{n}(z)|$ if $k\ge n >\ell$. Assume $\ell <N$. We have
$$
\sum^{N-1}_{n\ge \ell}|a_{n}|\frac{1-|f^{n}(z)|^{2}}{|f^{n}(z)|} \sum^{N}_{k>n} |a_{k}| |f^{k}(z)|\le r_{0}^{-1} \sum^{N-1}_{n\ge \ell}|a_{n}| \sum^{N}_{k>n}|a_{k}|c_{0}^{k-n}.
$$
Writting $j=k-n$ and applying Cauchy--Schwarz's inequality, last double sum can be bounded by
$$
\sum^{N- \ell}_{j=1}c_{0}^{j} \sum^{N-j}_{n\ge \ell} |a_{n}| |a_{n+j}| \le (1-c_{0})^{-1} \sum^{N}_{n\ge \ell}|a_{n}|^{2}.
$$
Since $|f^{n}(z)|\le |f^{\ell}(z)|\le r_{0}$ for $n\ge \ell$, last expression above is bounded by
$$
(1-r_{0}^{2})^{-1} (1-c_{0})^{-1}\sum^{N}_{n\ge \ell} |a_{n}|^{2} (1-|f^{n}(z)|^{2}).
$$
Hence we only need to estimate $B_{1}+B_{2}$ where
\begin{align*}
B_{1}&= \sum^{\ell-1}_{n=1} |a_{n}| \frac{1-|f^{n}(z)|^{2}}{|f^{n}(z)|} \sum^{\ell}_{k>n}|a_{k}| |f^{k}(z)|,\\*[5pt]
B_{2}&= \sum^{\ell-1}_{n=1} |a_{n}| \frac{1-|f^{n}(z)|^{2}}{|f^{n}(z)|} \sum^{N}_{k>\ell}|a_{k}| |f^{k}(z)|.
\end{align*}
Since $|f^{k}(z)|\le |f^{n}(z)|$ for $k\ge n$, writting $j=k-n$ we obtain
$$
B_{1}\le \sum^{\ell-1}_{n=1} |a_{n}| (1-|f^{n}(z)|^{2}) \sum^{\ell}_{k>n}|a_{k}| \le
\sum^{\ell}_{j=1}\sum^{\ell-j}_{n=1}|a_{n}| |a_{n+j}| (1-|f^{n}(z)|^{2}).
$$
Since $|f^{n}(z)|\ge r_{0}$ for $n<\ell$, part~(a) of Lemma~\ref{lem2.1} provides a constant $0<c_{1}=c_{1}(f)<1$ such that $1-|f^{n}(z)|\le c_{1}^{j} (1-|f^{n+j}(z)|)$ for any $n<n+j\le \ell$. Hence
$$
B_{1}\le 2\sum^{\ell}_{j=1}c_{1}^{j/2} \sum^{\ell-j}_{n=1} |a_{n}| (1-|f^{n}(z)|)^{1/2} |a_{n+j}|(1-|f^{n+j}(z)|)^{1/2}.
$$
Applying Cauchy--Schwarz's inequality, we deduce
$$
B_{1}\le 2 (1-c_{1}^{1/2})^{-1} \sum^{\ell}_{n=1}|a_{n}|^{2} (1-|f^{n}(z)|).
$$
Finally we estimate $B_{2}$. Note that Corollary~\ref{coro2.2} gives $|f^{k}(z)|\le r_{0}^{-1} c_{0}^{k-\ell}|f^{\ell}(z)|$ if $k\ge \ell$. Since $|f^{n}(z)|\ge |f^{\ell}(z)|$ for $n\le\ell$, we have
$$
B_{2}\le r_{0}^{-1} \sum^{\ell-1}_{n=1}|a_{n}| (1-|f^{n}(z)|^{2})\sum^{N}_{k>\ell }|a_{k}|c_{0}^{k-\ell}.
$$
Cauchy--Schwarz's inequality gives
\begin{equation}\label{eq2.4}
\sum^{N}_{k>\ell}|a_{k}|c_{0}^{k-\ell} \le (1-c_{0}^2)^{-1/2} \left(\sum^{N}_{k>\ell}|a_{k}|^{2}\right)^{1/2}
\end{equation}
and
$$
\sum^{\ell-1}_{n=1}|a_{n}| (1-|f^{n}(z)|)\le \left(\sum^{\ell-1}_{n=1}|a_{n}|^{2} (1-|f^{n}(z)|)\right)^{1/2} \left(\sum^{\ell-1}_{n=1}(1-|f^{n}(z)|)\right)^{1/2}.
$$
Since $1-|f^{n}(z)|\le c_{1}^{\ell-n}$ for $1\le n<\ell$, last sum is bounded by~$(1-c_{1})^{-1}$ and we deduce
\begin{equation}\label{eq2.51}
\sum^{\ell-1}_{n=1}|a_{n}| (1-|f^{n}(z)|)\le (1-c_{1})^{-1/2} \left(\sum^{\ell-1}_{n=1}|a_{n}|^{2} (1-|f^{n}(z)|)\right)^{1/2}.
\end{equation}
Now applying \eqref{eq2.4} and \eqref{eq2.51} we deduce
\begin{equation*}
\begin{split}
B_{2}&\le c_{2} \left(\sum_{n=1}^{\ell-1} |a_{n}|^{2} (1-|f^{n}(z)|^{2})\right)^{1/2}  \left(\sum^{N}_{k>\ell}|a_{k}|^{2}\right)^{1/2}\\*[5pt]
&\le \frac{c_{2}}{2}\left(\sum^{\ell-1}_{n=1}|a_{n}|^{2} (1-|f^{n}(z)|^{2})+\sum^{N}_{k>\ell}|a_{k}|^{2}\right),
\end{split}
\end{equation*}
where $c_{2}=2 r_0^{-1} (1-c_{0}^2)^{-1/2}(1-c_{1})^{-1/2}$. Finally observe that $|f^{k}(z)|\le r_{0}$ for $k>\ell$ to deduce
$$
\sum^{N}_{k>\ell}|a_{k}|^{2}\le (1-r_{0}^{2})^{-1} \sum^{N}_{k>\ell}|a_{k}|^{2} (1-|f^{k}(z)|^{2}).
$$
This finishes the proof if $\ell <N$. Assume now $\ell \geq N$, that is $|f^k (z)| > r_0$ for any $k<N$.  Then we argue as in the estimate of $B_1$ replacing $\ell$ by $N$, to obtain
$$
 \sum^{N-1}_{n=1} |a_{n}| \frac{1-|f^{n}(z)|^{2}}{|f^{n}(z)|} \sum^{N}_{k>n}|a_{k}| |f^{k}(z)|  \leq 2 (1- c_1^{1/2})^{-1} \sum_{n=1}^{N} |a_{n}|^{2} (1-|f^{n}(z)| ).
$$
This finishes the proof. 
\end{proof}

The following well known auxiliary result plays a fundamental role in the classical work of Paley and Zygmund, as well as in the study of pointwise convergence of random series of functions (\cite{Ka}). Its short proof is included for the sake of completeness.

\begin{lemma}[Paley--Zygmund inequality]\label{lem2.4}
Let $(X,\Omega,d\mu)$ be a probability space and $Z\colon X\to [0,\infty)$ be a positive square integrable random variable. Then for any $0<\lambda <1$, we have
$$
\mu\left\{x\in X:Z(x)>\lambda\int_{X}Z\,d\mu\right\}\ge (1-\lambda)^{2} \frac{\left(\displaystyle\int_{X}Z\,d\mu\right)^{2}}{\displaystyle\int_{X}Z^{2}\,d\mu}.
$$
\end{lemma}

\begin{proof}
We can assume $\displaystyle\int_{X}Z\,d\mu=1$. Let $W_{\lambda}$ denote the indicator function of the set~$\{x\in X: Z(x) \leq \lambda\}$. Cauchy--Schwarz's inequality gives
$$
1=\int_{X}ZW_{\lambda}\,d\mu +\int_{X}Z(1-W_{\lambda})\,d\mu\le \lambda +\left( \int_{X}Z^{2}\,d\mu\right)^{1/2} \mu \{x\in X:Z(x)>\lambda\}^{1/2}.
$$
\end{proof}

We are now ready to prove Theorem \ref{theo1.1} stated in the Introduction.

\begin{proof}[Proof of Theorem~\ref{theo1.1}]
It is obvious that (a) implies (b). We start proving that (c) implies (a). By Lemma \ref{lem2.3}, the series $\sum a_{n}f^{n}$ converges in $\mathbb{H}^{2}(\mathbb{D})$. Let $F(z)=\sum\limits_{n=1}^{\infty} a_{n}f^{n}(z)$, $z\in\mathbb{D}$. Hence the non-tangential limit
$$
F(\xi) =\lim_{{z {\to}_{\nless} \xi} }F(z)
$$
exists for almost every $\xi\in\partial\mathbb{D}$. For almost every $\xi\in \partial\mathbb{D}$ we will construct a sequence of points~$\{z_{N}=z_{N}(\xi): N=1,2,\dotsc\}$ tending non-tangentially to~$\xi$, such that for any $\varepsilon>0$ there exists an integer $N_{0}=N_{0}(\varepsilon,\xi)>0$ satisfying 
\begin{equation}\label{eq2.5}
\left|\sum^{N}_{n=1}a_{n} f^{n}(\xi) - F(z_{N})\right|<\varepsilon,\quad N \geq N_{0}.
\end{equation}
It is clear that \eqref{eq2.5} implies the statement in~(a).~By Corollary~\ref{coro2.2}, there exist constants $0<r_{0}=r_{0}(f)<1$, $0<c_{0}=c_{0}(f)<1$ such that
\begin{equation}\label{eq2.6}
|f^{n}(z)|\le r_{0}^{-1}c_{0}^{n} |z|,\quad |z|\le r_{0},\quad n=1,2,\dotsc
\end{equation}
Fix $N\ge 1$.~Write $F_{N}=\sum\limits^{N}_{n=1}a_{n}f^{n}$.~Since $f(0)=0$ and $f^{N}$ is inner, there exists a subset $S=S_{N}\subset\partial\mathbb{D}$ with $m(\partial\mathbb{D}\backslash S)=0$ such that for any $\xi\in S$ there exists $0<r_{N}=r_{N}(\xi)< 1$ such that $|f^{N} (r_{N}\xi)|=r_{0}$. Since $f^{N}$ tends to~$0$ uniformly on compact sets of~$\mathbb{D}$ as~$N\to\infty$, we deduce that $r_{N}\to 1$ as $N\to\infty$. Note also that $|f^{j}(r_{N}\xi)|\ge |f^{N}(r_{N}\xi)|=r_{0}$ for $1\le j\le N$. For $\xi\in S$ consider the arc~$I(N,\xi)=\{z\in\partial\mathbb{D}:|z-\xi|<1-r_{N}\}$. Apply Vitali's Covering Lemma (see~\cite[p.~27]{EG}) to obtain a subcollection~$\{I(N,\xi_{k}):k=1,2,\dotsc\}$ of pairwise disjoint arcs such that $m(\partial\mathbb{D}\backslash\cup 5I(N,\xi_{k}))=0$. Here $5I$ denotes the arc in the unit circle having the same center than $I$ and with $m(5I) = 5 m(I)$. Given $\varepsilon>0$ consider the set
$$
E(N)=\bigcup_{k}\{ \xi\in 5 I(N,\xi_{k}): |F_{N}(\xi)-F_{N}(\xi_{k})|\ge\varepsilon\}
$$
and
$$
E=\bigcap_{\ell\ge 1}\bigcup_{N\ge\ell}E(N).
$$
We first show that $m(E)=0$. By the classical Borel--Cantelli Lemma, it is sufficient to show
\begin{equation}\label{eq2.7}
\sum^{\infty}_{N=1}m(E(N))<\infty.
\end{equation}
Observe that Theorem~\ref{theo1.2} gives that $F\in\operatorname{BMOA}$. Since $P(r_{N}	\xi_{k},\xi)$ is comparable to  $m(I(N,\xi_{k}))^{-1}$ for $\xi\in 5I(N,\xi_{k})$, there exists a constant~$c_{1}=c_{1}(f)>0$ such that for any $k,N\ge 1$, we have
$$
\frac{1}{m(I(N,\xi_{k}))}\int_{5I(N,\xi_{k})} |F_{N}(\xi)-F_{N}(r_{N}\xi_{k})|^{2}\,dm(\xi)\le
c_{1}\sum^{N}_{n=1}|a_{n}|^{2} (1-|f^{n}(r_{N}\xi_{k})|).
$$
By part~(a) of Lemma~\ref{lem2.1} and the choice of~$r_{N}$, there exists a constant $0<c<1$ such that $1-|f^{n}(r_{N}\xi_{k})|\le c^{N-n}(1-r_{0})$, $n=1,\dotsc,N$. Hence
$$
\frac{1}{m(I(N,\xi_{k}))}\int_{5I(N,\xi_{k})}|F_{N}(\xi)-F_{N}(r_{N}\xi_{k})|^{2}\,dm(\xi)\le c_{1}(1-r_{0})\sum^{N}_{n=1}|a_{n}|^{2}c^{N-n}, \quad k,N\ge 1. 
$$
Given $\varepsilon >0$, we deduce that for any $k,N\ge 1$, we have 
$$
\frac{1}{m(I(N,\xi_{k}))} m(\{\xi\in 5I(N,\xi_{k}):|F_{N}(\xi)-F_{N}(r_{N}\xi_{k})|>\varepsilon\})\le \frac{c_{1}(1-r_{0})}{\varepsilon^{2}} \sum^{N}_{n=1}|a_{n}|^{2}c^{N-n}.
$$
Since $\{I(N,\xi_{k}): k=1,2, \dotsc\}$ are pairwise disjoint, we deduce
$$
m(E(N))\le \frac{c_{1}(1-r_{0})}{\varepsilon^{2}} \sum^{N}_{n=1}|a_{n}|^{2} c^{N-n}.
$$
Then
$$
\sum^{\infty}_{N=1}m(E(N)) \le \frac{c_{1}(1-r_{0})}{\varepsilon^{2}} \sum^{\infty}_{N=1} \sum^{N}_{n=1}|a_{n}|^{2}c^{N-n} =  \frac{c_{1}(1-r_{0})}{\varepsilon^{2}(1-c)} \sum^{\infty}_{n=1}|a_{n}|^{2}.
$$
This proves \eqref{eq2.7} and shows that $m(E)=0$. Finally for almost every $\xi\in \partial\mathbb{D}\backslash E$ we will construct the sequence $\{z_{N}=z_{N}(\xi) \}$ for which \eqref{eq2.5} holds. Since $m(\partial\mathbb{D}\backslash\cup 5I(N,\xi_{k}))=0$, for almost every $\xi\in\partial\mathbb{D}\backslash E$ there exists $N_{0}=N_{0}(\xi)>0$ such that for any~$N>N_{0}$ there exists $\xi_{k}\in\partial\mathbb{D}$ with $\xi\in 5I(N,\xi_{k})$ and
\begin{equation}\label{eq2.8}
|F_{N}(\xi)-F_{N}(r_{N}\xi_{k} )|<\varepsilon.
\end{equation}
The choice of $r_{N}$ and \eqref{eq2.6} gives that
$$
|f^{n}(r_{N}\xi_{k})|\le c_{0}^{n-N},\quad n\ge N.
$$
Hence
\begin{equation}\label{eq2.9}
\sum_{n>N}|a_{n}| |f^{n}(r_{N}\xi_{k})|\le \sum_{n>N}|a_{n}|c_{0}^{n-N}\le \left(1-c_{0}^2 \right)^{-1/2} \left(\sum_{n>N}|a_{n}|^{2}\right)^{1/2}.
\end{equation}
Pick $z_N = z_N (\xi) = r_N \xi_k$, where $\xi_k$ is chosen such that $\xi\in 5I(N,\xi_{k})$. Now \eqref{eq2.8} and \eqref{eq2.9} imply \eqref{eq2.5}. Note that since $\xi\in 5I(N, \xi_{k})$ and $r_{N}\to 1$, the points~$z_{N}$ converge non-tangentially to~$\xi$ as $N$ tends to $\infty$.

We now prove that (b) implies (c). Given two sets $A,B\subset\partial\mathbb{D}$ we use the notation $A\overset{\text{a.e.}}{=}B$ if $A$ and $B$ differ at most on a set of Lebesgue measure zero. Consider the set
$$
A=\left\{\xi\in\partial\mathbb{D}:\sup_{N}\left|\sum^{N}_{n=1}a_{n}f^{n}(\xi)\right|<\infty\right\}.
$$
Note that $f^{-k} (f^{k}(A))\overset{\text{a.e.}}{=} A$ for any $k=1,2, \dotsc$ Since the mapping $f\colon \partial\mathbb{D}\to \partial\mathbb{D}$ is exact (\cite{Ne}) we deduce that $m(A)=0$ or $m(A)=1$. Our assumption gives $m(A)=1$. Write
$$
F_{N}=\sum^{N}_{n=1}a_{n}f^{n}
$$
and $s_{N}=\|F_{N}\|_{2}$. By Lemma~\ref{lem2.3} there exists a constant~$c=c(f)>1$ such that
\begin{equation}\label{eq2.10}
c^{-1} \sum^{N}_{n=1}|a_{n}|^{2}\le s^{2}_{N}\le c\sum^{N}_{n=1}|a_{n}|^{2},\quad N=1,2,\dotsc
\end{equation}
Now the Paley--Zygmund inequality stated in Lemma~\ref{lem2.4} gives that for any $0<\lambda<1$ we have
$$
m\{\xi\in\partial\mathbb{D}:|F_{N}(\xi)|^{2}\ge \lambda s_{N}^{2}\}\ge (1-\lambda)^{2}\frac{s_{N}^{4}}{\displaystyle\int_{\partial\mathbb{D}}|F_{N}|^{4}\,dm}.
$$
It was proved in \cite[part (c) of Theorem~9]{NS2} that there exists a constant~$c_{1}=c_{1}(f)>0$ such that
$$
\int_{\partial\mathbb{D}}|F_{N}|^{4}\,dm \le c_{1} s_{N}^{4}.
$$
Hence 
$$
m\{\xi\in\partial\mathbb{D}: |F_{N}(\xi)|^{2}\ge \lambda s_{N}^{2}\}\ge c_{1}^{-1} (1-\lambda)^{2},\quad 0<\lambda <1.
$$
Since $m(A)=1$, we deduce that $s_{N}$ is bounded. Applying \eqref{eq2.10} we deduce $\sum\limits_{n}|a_{n}|^{2}<\infty$.
\end{proof}

We are now ready to prove Corollaries~\ref{coro1.3} and~\ref{coro1.4}.

\begin{proof}[Proof of Corollary~\ref{coro1.3}]
By Lemma \ref{lem2.3}, the series $\sum a_{n}f^{n}$ converges in $\mathbb{H}^{2}(\mathbb{D})$. Let $F=\sum\limits_{n\ge 1} a_{n}f^{n}$. Theorem~\ref{theo1.2} gives that $F\in\operatorname{BMOA}(\mathbb{D})$ and $\|F\|^{2}_{\operatorname{BMOA}(\mathbb{D})}\le C \sum\limits_{n\ge 1}|a_{n}|^{2}$. The converse estimate follows easily from Lemma \ref{lem2.3} which gives 
$$
\|F\|_{\operatorname{BMOA}(\mathbb{D})}^{2}\ge \|F\|_{2}^{2}\ge \frac{1-|f'(0)|}{1+|f'(0)|} \sum\limits_{n=1}^{\infty} |a_{n}|^{2} . 
$$
The identity~\eqref{eq1.1} follows from \eqref{eq2.5}. This finishes the proof of (a). Assume now that $f$ is a finite Blaschke product. Fix a number $0<r<1$ and let $0<c=c(r,f)<1$ be the constant appearing in part~(a) of Lemma~\ref{lem2.1}. For $z\in\mathbb{D}$ let $N=N(z)$ be the smallest positive integer such that $|f^{N}(z)|\le r$. Note that $|f^{n}(z)|\le r$ for $n\ge N$ and $1-|f^{n}(z)|\le c^{N-n}$ for $1\le n\le N$. Since $f$ is a finite Blaschke product, we have
\begin{equation}\label{eq2.11}
\lim_{|z|\to 1}N(z)=+\infty.
\end{equation}
Now
$$
\sum^{\infty}_{n=1}|a_{n}|^{2} (1-|f^{n}(z)|)\le \sum^{N}_{n=1}|a_{n}|^{2}c^{N-n}+\sum^{\infty}_{n=N+1}|a_{n}|^{2}
$$
and \eqref{eq2.11} gives
$$
\lim_{|z|\to 1}\sum^{\infty}_{n=1}|a_{n}|^{2} (1-|f^{n}(z)|^{2})=0.
$$
Theorem~\ref{theo1.2} finishes the proof.
\end{proof}

\begin{proof}[Proof of Corollary~\ref{coro1.4}]
Let $F=\sum\limits_{n=1}^{\infty}a_{n}f^{n}$ and $s^{2}=\sum\limits_{n=1}^{\infty}|a_{n}|^{2}$. We first prove the upper estimate.
Consider the distribution function $\Phi(\lambda)=m\{\xi\in\partial\mathbb{D}: |F(\xi)|>\lambda\}$, $\lambda>0$. Corollary~\ref{coro1.3} and the John--Nirenberg Theorem give that there exist universal constants $A,B>0$ such that
$$
\Phi(\lambda)\le A e^{-B\lambda/s},\quad \lambda>0.
$$
Then
$$
\|F\|_{p}^{p}=\int_{0}^{\infty} p\lambda^{p-1} \Phi (\lambda)\,d\lambda\le C(p)s^{p},
$$
where $C(p)$ is a constant depending on $A$, $B$ and $p$. The lower estimate follows from the following standard duality argument. We can assume $p<2$. By H\"older's inequality
$$
\int_{\partial\mathbb{D}}|F|^{2}\,dm\le \left(\int_{\partial\mathbb{D}} |F|^{p}\,dm\right)^{1/p}\left(
\int_{\partial\mathbb{D}} |F|^{q}\,dm\right)^{1/q},
$$
where $p^{-1}+q^{-1}=1$. By Lemma~\ref{lem2.3} and the upper estimate we have already proved, there exists a constant~$c=c(f,p)>0$ such that
$$
\sum^{\infty}_{n=1}|a_{n}|^{2}\le c\left( \int_{\partial\mathbb{D}} |F|^{p}\,dm\right)^{1/p} \left(\sum^{\infty}_{n=1}|a_{n}|^{2}\right)^{1/2}.
$$
This finishes the proof.
\end{proof}

\section{Uniform estimates}\label{sec3}

Given $F\in L^{1}(\partial \mathbb{D})$ and $z\in \mathbb{D}$, let $F(z)$ denote the value of its harmonic extension at~$z$, that is, 
$$
F(z)=\int_{\partial \mathbb{D}}F(\xi)P(z,\xi)\,dm (\xi).
$$
If $F$ is the characteristic function of a measurable set~$E\subset\partial \mathbb{D}$, the corresponding function is called the harmonic measure of the set~$E$ from the point~$z\in \mathbb{D}$ and will be denoted by $F(z) = w (z, E)$. A function $F\in L^{1}(\partial \mathbb{D})$ is in the space~$\operatorname{BMO}(\partial \mathbb{D})$ if
$$
\|F\|_{\operatorname{BMO}(\partial \mathbb{D})}^{2}=\sup_{z\in \mathbb{D}}\int_{\partial \mathbb{D}}|F(\xi)-F(z)|^{2}P(z,\xi)\,dm(\xi)<\infty.
$$
Equivalently, $F\in L^{1}(\partial \mathbb{D})$ is in $\operatorname{BMO}(\partial \mathbb{D})$ if and only if there exists a constant~$C=C(F)>0$ such that for any arc~$I\subset \partial \mathbb{D}$, we have
$$
\frac{1}{m(I)}\int_{I}|F-F_{I}|^{2}\,dm\le C.
$$
Here $F_{I}=m(I)^{-1}\int_{I}F\,dm$ denotes the mean of~$F$ on~$I$. See Chapter VI of~\cite{Ga}. For $z\in \mathbb{D}$ let $ \tau_{z}$  be the automorphism of~$\mathbb{D}$ given by $\tau_{z}(w)=(w-z)(1-\overline{z}w)^{-1}
$, $w\in \mathbb{D}$. The John--Nirenberg Theorem applied to $F\circ \tau_{z}-F(z)$ provides two universal constants~$A,B>0$ such that
\begin{equation}\label{eqJN}
w(z,\{\xi\in\partial \mathbb{D}: |F(\xi)-F(z)|>\lambda\})\le A e^{-B\lambda/\|F\|_{\operatorname{BMO}(\partial \mathbb{D})}},\quad \lambda>0. %\tag{JN}
\end{equation}
See \cite{Ba}. The proof of Theorem \ref{theo1.5} is based on a careful study of the oscillation of the partial sums of $\sum a_n f^n$. We start with an auxiliary result which holds for any $BMO$ function. 

\begin{lemma}\label{lem3.1}
For any $0<c<1$ there exists a constant~$0< \delta=\delta(c) <1$ such that the following statement holds. Let $F$ be a real valued function defined on~$\partial \mathbb{D}$ with $\|F\|_{\operatorname{BMO}(\partial \mathbb{D})}=1$. Let $z\in\mathbb{D}$  such that
\begin{equation}\label{eq3.1}
\int_{\partial\mathbb{D}}|F(\xi)-F(z)|^{2}P(z,\xi)\,dm(\xi)\ge c^{2}.
\end{equation}
Then $w(z,\{\xi\in\partial\mathbb{D}: F(\xi)-F(z)\ge \delta\})\ge \delta$.
\end{lemma}

\begin{proof}
Given $0\le a<b$, let $E(a,b)$ denote the set of points~$\xi\in\partial\mathbb{D}$ such that $a\le |F(\xi)-F(z)|\le b$. Apply the John--Nirenberg estimate~\eqref{eqJN} to find a constant~$c_{1}=c_{1}(c)>0$, such that
$$
\int_{\partial\mathbb{D}\backslash E(0,c_{1})} |F(\xi)-F(z)|^{2}P(z,\xi)\,dm(\xi)\le \frac{c^{2}}{4}.
$$
Since
$$
\int_{E(0,c/2)}|F(\xi)-F(z)|^{2}P(z,\xi)\,dm(\xi)\le \frac{c^{2}}{4},
$$
we deduce that
$$
\int_{E(c/2,c_{1})}|F(\xi)-F(z)|^{2}P(z,\xi)\,dm(\xi)\ge\frac{c^{2}}{2}.
$$
Hence
\begin{equation}\label{eq3.2}
w\left(z,\left\{\xi\in\partial \mathbb{D}: |F(\xi)-F(z)|\ge \frac{c}{2}\right\}\right)\ge \frac{c^{2}}{2c_{1}^{2}}.
\end{equation}
Apply the John--Nirenberg estimate~\eqref{eqJN} to find a constant~$c_{2}=c_{2}(c)>0$ such that
\begin{equation}\label{eq3.3}
\int_{\partial\mathbb{D}\backslash E(0,c_{2})} |F(\xi)-F(z)|P(z,\xi)\,dm(\xi)\le \frac{c^{3}}{16c_{1}^2}.
\end{equation}
Let $0<\delta=\delta(c)<c/2$ be a small constant to be fixed later.~We will show that if $\delta >0$ is conveniently chosen, then
\begin{equation}\label{eq3.4}
w(z,\{\xi:F(\xi)-F(z)>\delta\})\ge\min \left\{\frac{c^{3}}{16c_{1}^{2}c_{2}},\frac{c^{2}}{16c_{1}^{2}}\right\}.
\end{equation}
We argue by contradiction. Assume
\begin{equation}\label{eq3.5}
w(z,\{\xi:F(\xi)-F(z)>\delta\})\le\min \left\{\frac{c^{3}}{16c_{1}^{2}c_{2}},\frac{c^{2}}{16c_{1}^{2}}\right\}.
\end{equation}
Apply \eqref{eq3.2} to deduce that $w(z,\{\xi:F(\xi)-F(z)<-c/2\})\ge 7c^{2}/16c_{1}^{2}$. Hence
\begin{equation}\label{eq3.6}
\int_{\{\xi:F(\xi)-F(z)\le 0\}} (F(\xi)-F(z))P(z,\xi)\,dm(\xi)\le \frac{-7c^{3}}{32c_{1}^{2}}.
\end{equation}
For $0< \gamma < \tau$ consider the set $G(\gamma, \tau) = \{\xi \in \partial \mathbb{D} : \gamma \leq F(\xi )- F(z) \leq \tau\}$. We have 
\begin{equation*}
\begin{split}
\int_{\{\xi:F(\xi)-F(z)>0\}} (F(\xi)-F(z))P(z,\xi)\,dm(\xi)&\le\delta +\int_{G(\delta , c_2)}
(F(\xi)-F(z))P(z,\xi)\,dm(\xi) \\
&\quad+\int_{\{\xi:F(\xi)-F(z)>c_{2}\}}(F(\xi)-F(z))P(z,\xi)\,dm(\xi).
\end{split}
\end{equation*}
The choice~\eqref{eq3.3} of $c_{2}$ gives that last integral is bounded by $c^{3}/16c_{1}^{2}$. Moreover by~\eqref{eq3.5}, we have 
$$
\int_{G(\delta , c_2)} (F(\xi)-F(z))P(z,\xi)\,dm(\xi) \le c_{2}w(z,\{\xi:F(\xi)-F(z)>\delta\})\le \frac{c^{3}}{16c_{1}^{2}}.
$$
We deduce
\begin{equation}\label{eq3.7}
\int_{\{\xi: F(\xi)-F(z)>0\}} (F(\xi)-F(z))P(z,\xi)\,dm(\xi) \le\delta +\frac{c^{3}}{8c_{1}^{2}}.
\end{equation}
Choosing $0<\delta <c^{3}/16c_{1}^{2}$ we observe that \eqref{eq3.6} and \eqref{eq3.7} contradict the identity 
$$
\int_{\partial\mathbb{D}}  (F(\xi)-F(z))P(z,\xi)\,dm(\xi) =0.
$$
Hence choosing $0<\delta<\min\{c/2,c^{3}/16c_{1}^{2}\}$, estimate~\eqref{eq3.4} holds. This finishes the proof.
\end{proof}

Our next auxiliary result says that $F=\sum\limits_{n=M}^{N}a_{n}f^{n}$ satisfies condition~\eqref{eq3.1} in Lemma~\ref{lem3.1} if $|f^{M}(z)|$ is sufficiently small.

\begin{lemma}\label{lem3.2}
Let $f$ be an inner function with $f(0)=0$ which is not a rotation. Then there exists a constant~$ 0< \varepsilon =\varepsilon (f)< 1$ such that if $M<N$ are positive integers, $\{a_{n}:M \le n\le N\}$ are complex numbers and $z\in\mathbb{D}$ satisfies $|f^{M}(z)|\le\varepsilon$, we have
$$
\int_{\partial\mathbb{D}}\left|\sum^{N}_{n=M}a_{n}(f^{n}(\xi)-f^{n}(z))\right|^{2} P(z,\xi)\,dm(\xi)\ge \frac{1}{2}\, \frac{1-|f'(0)|}{1+|f'(0)|} \sum^{N}_{n=M}|a_{n}|^{2}.
$$
\end{lemma}

\begin{proof}
By continuity we can assume that $f^{n}(z)\ne 0$ for any~$n$. Write $F=\sum\limits_{n=M}^{N}a_{n}f^{n}$. In \eqref{eq2.3} we already noted that
\begin{equation}\label{eq3.8}
\begin{split}
\int_{\partial \mathbb{D}}|F(\xi)-F(z)|^{2}
P(z,\xi)\,dm(\xi)&=\sum^{N}_{n=M} |a_{n}|^{2}(1-|f^{n}(z)|^{2})\\
&\quad+2\operatorname{Re}\sum^{N-1}_{n=M}
\sum^{N}_{k>n}\overline{a}_{n}a_{k}(1-|f^{n}(z)|^{2}) \frac{f^{k}(z)}{f^{n}(z)}.
\end{split}
\end{equation}
The idea is to compare the expression above with
\begin{equation}\label{eq3.9}
\int_{\partial\mathbb{D}}|F(\xi)|^{2}\,dm(\xi)=\sum^{N}_{n=M}|a_{n}|^{2} +2\operatorname{Re}\sum^{N-1}_{n=M}\sum^{N}_{k>n}\overline{a}_{n}a_{k}f'(0)^{k-n}.
\end{equation}

Assume $f'(0)\ne 0$. For $k>n\ge M$ consider the function~$g=g_{k,n}$ defined by $g(w)=f^{k-n}(w)$, $w\in \mathbb{D}$. By part~(b) of Lemma~\ref{lem2.1}, there exists a constant~$0<r_{0}=r_{0}(f)<1$ such that $|g(w)|\le r_{0}^{-1} |f'(0)|^{k-n}|w|$ if $|w|<r_{0}$. Hence there exists a constant~$c(r_{0})>0$ such that
\begin{equation}\label{eq3.10}
\left|\frac{g(w)}{w}-g'(0)\right|\le c(r_{0})|f'(0)|^{k-n}|w|,\quad |w|\le \frac{r_{0}}{2}.
\end{equation}
Let $0<\varepsilon < r_{0}/2$ be a constant to be fixed later and assume $|f^{M}(z)|\le \varepsilon $. Taking $w=f^{n}(z)$ in \eqref{eq3.10} we obtain
$$
\left|\frac{f^{k}(z)}{f^{n}(z)}-f'(0)^{k-n}\right| \le c(r_{0})|f'(0)|^{k-n} |f^{n}(z)|\le \varepsilon c(r_{0})|f'(0)|^{k-n}, \quad k>n \geq M.  
$$
By part (b) of Lemma 2.1, we also have $|f^{k}(z)|\le r_{0}^{-1}|f'(0)|^{k-M}\varepsilon $ and $|f^{n}(z)|\le r_{0}^{-1}|f'(0)|^{n-M}\varepsilon$ if $k \ge M$ and $n\ge M$. Then
\begin{equation*}
\begin{split}
&\left|\sum^{N-1}_{n=M} \sum^{N}_{k>n}\overline{a}_{n}a_{k}\left( (1-|f^{n}(z)|^{2})\frac{f^{k}(z)}{f^{n}(z)}-f'(0)^{k-n}\right)\right|\\*[5pt]
\le
&\sum^{N-1}_{n=M}\sum^{N}
_{k>n}|a_{n}| |a_{k}| (\varepsilon c(r_{0})|f'(0)|^{k-n}+r_{0}^{-2}\varepsilon^{2}|f'(0)|^{k+n-2M}).
\end{split}
\end{equation*}
As in the proof of Theorem \ref{theo1.2}, writting $j =k-n$ and applying Cauchy-Schwarz inequality, we find a constant $c(f)>0$ such that
\begin{align*}
&\sum^{N-1}_{n=M} \sum^{N}_{k>n}|a_{n}| |a_{k}| |f'(0)|^{k-n}\le
c(f)\sum^{N}_{n=M}|a_{n}|^{2},\\*[5pt]
&\sum^{N-1}_{n=M} \sum^{N}_{k>n}|a_{n}| |a_{k}| |f'(0)|^{k+n-2M}\le
c(f)\sum^{N}_{n=M}|a_{n}|^{2}.
\end{align*}
We deduce that there exists a constant $c(f,r_{0})>0$ such that
$$
\left|\sum^{N-1}_{n=M}\sum^{N}_{k>n}\overline{a}_{n}a_{k}\left( (1-|f^{n}(z)|^{2}) \frac{f^{k}(z)}{f^{n}(z)}-f'(0)^{k-n}\right)\right|\le c(f,r_{0})\varepsilon \sum^{N}_{n=M}|a_{n}|^{2}.
$$
Applying \eqref{eq3.8} and \eqref{eq3.9}, we deduce
$$
\left|\int_{\partial\mathbb{D}}|F(\xi)-F(z)|^{2}P(z,\xi)\,dm(\xi)-\int_{\partial\mathbb{D}} |F(\xi)|^{2}\,dm(\xi)\right|\le \varepsilon^{2} \sum^{N}_{n=M}|a_{n}|^{2}+2c(f,r_{0})\varepsilon \sum^{N}_{n=M}|a_{n}|^{2}.
$$
Recall that
$$
\int_{\partial\mathbb{D}}|F(\xi)|^{2}\,dm(\xi)\ge \frac{1-|f'(0)|}{1+|f'(0)|}\sum^{N}_{n=M}|a_{n}|^{2}.
$$
Choose $0<\varepsilon <r_{0} / 2 $ small enough so that $\varepsilon^{2}+2c(f,r_{0})\varepsilon <(1-|f'(0)|) (1+|f'(0)|)^{-1}/2$ and deduce
$$
\int_{\partial\mathbb{D}}|F(\xi)-F(z)|^{2} P(z,\xi)\,dm(\xi)\ge \frac{1}{2}\, \frac{1-|f'(0)|}{1+|f'(0)|}\sum^{N}_{n=M}|a_{n}|^{2}.
$$
This finishes the proof if $f'(0) \neq 0$. Assume now $f'(0)=0$.~Apply part~(c) of Lemma~\ref{lem2.1} to the inner function~$f^{k-n}$ to obtain $|f^{k}(z)|\le |f^{n}(z)|^{2^{k-n}}$, $k>n$. Let $ 0 < \varepsilon < 1$ be a (small) constant to be fixed later and assume $|f^M (z)| \leq \varepsilon$. If $k > n\ge M$, we have 
$$
\frac{|f^k (z)|}{|f^n (z)|} \leq |f^n (z)|^{2^{k-n} - 1} \leq \varepsilon^{2^{k-n} - 1} \leq  \varepsilon^{k-n}. 
$$
Writing $j=k-n$ and applying Cauchy-Schwarz's inequality, we obtain 
$$
\left|\sum^{N-1}_{n=M}\overline{a}_{n}\sum^{N}_{k>n}a_{k}(1-|f^{n}(z)|^{2})\frac{f^{k}(z)}{f^{n}(z)}\right|\le
\sum^{N-1}_{n=M}\sum^{N}_{k>n}|a_{n}| |a_{k}|\varepsilon^{k-n}\le \varepsilon (1-\varepsilon)^{-1} \sum^{N}_{n=M} |a_{n}|^{2}.
$$
As before, applying \eqref{eq3.8} and \eqref{eq3.9} we deduce
$$
\left|\int_{\partial \mathbb{D}}|F(\xi)-F(z)|^{2}P(z,\xi)\,dm(\xi)-\int_{\partial\mathbb{D}}|F(\xi)|^{2}\,dm(\xi)\right|\le (\varepsilon^{2}+2\varepsilon (1-\varepsilon)^{-1})\sum^{N}_{n=M}|a_{n}|^{2}.
$$
Since
$$
\int_{\partial\mathbb{D}}|F(\xi)|^{2}\,dm(\xi)=\sum^{N}_{n=M}|a_{n}|^{2}
$$
we only need to pick $\varepsilon >0$ small enough  so that $\varepsilon^{2}+ 2 \varepsilon (1- \varepsilon)^{-1}<1/2$.
\end{proof}

Next auxiliary result is an easy consequence of Lemmas~\ref{lem3.1} and \ref{lem3.2}.

\begin{lemma}\label{lem3.3}
Let $f$ be an inner function with $f(0)=0$ which is not a rotation.~Then there exist two constants~$0 < \varepsilon =\varepsilon (f)< 1$ and $0 < c =c(f) < 1$ such that the following statement holds. Let $M<N$ be positive integers, $z\in\mathbb{D}$ satisfying $|f^{M}(z)|<\varepsilon$ and $\{a_{n}:M\le n\le N\}$ a set of complex numbers. Then there exists a set~$E=E(z,\{a_{n}\}, f)\subset \partial\mathbb{D}$ with $w(z,E)\ge c$ such that
$$
\operatorname{Re} \sum^{N}_{n=M}a_{n}f^{n}(\xi)\ge c \left(\sum_{n=M}^{N}|a_{n}|^{2}\right)^{1/2},\quad \xi\in E.
$$
\end{lemma}

\begin{proof}
Note that for any analytic function~$G\in \mathbb{H}^{2}(\mathbb{D})$ with $G(z)=0$, we have
\begin{equation}\label{eq3.11}
\int_{\partial\mathbb{D}}(\operatorname{Re}G(\xi))^{2} P(z,\xi)\,dm(\xi)=\frac{1}{2}\int_{\partial\mathbb{D}}|G(\xi)|^{2}P(z,\xi)\,dm(\xi).
\end{equation}
Let $0<r_{0}=r_{0}(f)<1$ and $0<c_{0}=c_{0}(f)<1$ be the constants given in Corollary~\ref{coro2.2} and let $\varepsilon_{0}=\varepsilon_{0}(f)>0$ be the constant appearing in Lemma~\ref{lem3.2}. Let $0<\varepsilon<\min \{\varepsilon_{0},r_{0}\}$ be a (small) constant to be fixed later. Lemma~\ref{lem3.2} and the identity \eqref{eq3.11} applied to the function~$G=F-F(z)$, where $F=\sum\limits_{n=M}^{N} a_{n}f^{n}$, give that
$$
\int_{\partial\mathbb{D}}(\operatorname{Re}F(\xi)-\operatorname{Re}F(z))^{2}P(z,\xi)\,dm(\xi)\ge c(f)\sum^{N}_{n=M}|a_{n}|^{2},
$$
where $c(f)=4^{-1}(1-|f'(0)|)(1+|f'(0)|)^{-1}$. Applying Lemma~\ref{lem3.1} to the function 
$$
(\operatorname{Re}F)\|\mathrm{Re}\,F\|^{-1}_{\operatorname{BMO}(\partial\mathbb{D})} , 
$$ 
we find a constant~$\delta=\delta(f)>0$ and a set $E=E(z,\{a_{n}\}, f)\subset\partial\mathbb{D}$ with $w(z,E)\ge\delta$ such that
$$
\operatorname{Re}F(\xi)-\operatorname{Re}F(z)\ge\delta \|\mathrm{Re}\,F\|_{\operatorname{BMO}(\partial\mathbb{D})},\quad \xi\in E.
$$
By Corollary~\ref{coro1.3}, $\|\mathrm{Re}\, F\|^{2}_{\operatorname{BMO}(\partial\mathbb{D})}$ is comparable to $\sum\limits_{n=M}^{N}|a_{n}|^{2}$. Reducing $\delta>0$ if necessary, we can assume
\begin{equation}\label{eq3.12}
\operatorname{Re}F(\xi)-\operatorname{Re}F(z)\ge \delta\left(\sum^{N}_{n=M}|a_{n}|^{2}\right)^{1/2},\quad \xi\in E.
\end{equation}
We have $|f^{M}(z)|<\varepsilon$. Note that by Corollary~\ref{coro2.2}, we obtain $|f^{n}(z)|\le r_{0}^{-1}c_{0}^{n-M}\varepsilon$, $n\ge M$. Cauchy-Schwarz inequality yields 
$$
|F(z)|\le \varepsilon c(c_{0},r_{0})\left(\sum^{N}_{n=M}|a_{n}|^{2}\right)^{1/2},
$$
where $c(c_{0},r_{0})$ is a constant depending on~$c_{0}$ and~$r_{0}$. Pick $0 < \varepsilon < \min \{ \varepsilon_0 , r_0\}$ sufficiently small such that $\varepsilon  c(c_{0},r_{0}) \le \delta/2$. Taking $c =\delta/2$ the proof is completed.
\end{proof}

The next easy auxiliary result is stated for future reference.  
\begin{lemma}\label{lem3.4}
Let $F\in\operatorname{BMO}(\partial\mathbb{D})$ and $0<c<1$. Then there exists a constant  $\varepsilon=\varepsilon(c)>0$ such that if $I\subset \partial\mathbb{D}$ is an arc with $m ( \{ \xi\in I: \operatorname{Re}F(\xi)\ge c\|F\|_{\operatorname{BMO}(\partial\mathbb{D})} \} )\ge cm(I)$ and
\begin{equation}\label{3.99}
    \frac{1}{m(I)}\int_{I}|F-F_{I}|^{2}\,dm\le \varepsilon\|F\|^{2}_{\operatorname{BMO}(\partial\mathbb{D})},
\end{equation}
then $\operatorname{Re}(F_{I})\ge 2^{-1}c\|F\|_{\operatorname{BMO}(\partial\mathbb{D})}$.
\end{lemma}

\begin{proof}
Denote $E=\{\xi\in I:\operatorname{Re}F(\xi)\ge c\|F\|_{\operatorname{BMO}(\partial\mathbb{D})}\}$.~We argue by contradiction.~Assume $\operatorname{Re}(F_{I})\le 2^{-1}c\|F\|_{\operatorname{BMO}(\partial\mathbb{D})}$. Then the estimate \eqref{3.99} gives 
$$
\frac{c^{2}}{4}\|F\|^{2}_{\operatorname{BMO}(\partial\mathbb{D})}\frac{m(E)}{m(I)} \le \varepsilon \|F\|^{2}_{\operatorname{BMO}(\partial\mathbb{D})}
$$
and we deduce $m(E)\le 4\varepsilon c^{-2}m(I)$. Since $m(E)\ge cm(I)$, picking $0<\varepsilon < c^{3}/4$ one gets a contradiction.
\end{proof}

Given an arc~$I$ centered at $\xi\in\partial\mathbb{D}$ we denote by~$z(I)$ the point~$z(I)=(1-m(I))\xi$. Conversely, given a point~$z\in\mathbb{D}\backslash\{0\}$ let $I(z)$ be the arc in the unit circle such that $z(I(z))=z$. Given an arc~$I\subset\partial\mathbb{D}$ and a number~$0<c<1/m(I)$, we denote by~$cI$ the arc which has the same center as~$I$ and with $m(cI)=cm(I)$. Given an arc~$I\subset\partial\mathbb{D}$ we consider its dyadic decomposition~$\mathcal{D}(I)=\bigcup\limits_{n\ge 0}\mathcal{D}_{n}(I)$ where $\mathcal{D}_{n}(I)$ is the set of the $2^{n}$~pairwise disjoint subarcs of~$I$ of Lebesgue measure~$2^{-n}m(I)$.

\begin{lemma}\label{lem3.5}
Let $f$ be an inner function with $f(0)=0$ which is not a rotation. Then there exist constants~$0< \varepsilon =\varepsilon (f)<1$ and $0< c=c(f) < 1$ such that the following statement holds. Let $0<\gamma<1$ be a constant with $1-\gamma\le \varepsilon$, let $M<N$ be positive integers, $\{ a_{n}:M\le n\le N\}$ be complex numbers, and $z\in\mathbb{D}$ with $|f^{M}(z)|<\varepsilon$. Then there exist pairwise disjoint arcs~$\{I_{k}\}$ with $c^{-1}I_{k}\subset c^{-1}I(z)$ for any~$k$, such that
\begin{enumerate}
\item[(a)] For any $k=1,2,\dotsc$, we have $ \tau (1-\gamma)\le 1-|f^{N}(z(I_{k}))|\le 1-\gamma,$
where $\tau >0$ is a universal constant.

\item[(b)] $\sum m(I_{k})\ge c m(I(z))$.

\item[(c)] For any $k=1,2,\dotsc$, we have
$$
\frac{1}{|I_{k}|} \int_{I_{k}}\operatorname{Re}\sum^{N}_{n=M}a_{n}f^{n}\,dm\ge c \left(\sum^{N}_{n=M}|a_{n}|^{2}\right)^{1/2}.
$$
\end{enumerate}
\end{lemma}

\begin{proof}
Denote $F=\sum\limits_{n=M}^{N}a_{n}f^{n}$. Let $0<\varepsilon_{0}=\varepsilon_{0}(f)<1$, $0<c_{0}=c_{0}(f)<1$ be the constants given by Lemma~\ref{lem3.3}. Let $0<\varepsilon<\varepsilon_{0}(f)/2$ be a (small) constant to be fixed later. Assume $|f^{M}(z)|<\varepsilon$. Apply Lemma~\ref{lem3.3} to find a set~$E\subset\partial\mathbb{D}$ with $w(z,E)\ge c_{0}$ such that
\begin{equation}\label{eq3.13}
\operatorname{Re}F(\xi)\ge c_{0}\left(\sum^{N}_{n=M}|a_{n}|^{2}\right)^{1/2},\quad \xi\in E.
\end{equation}
Reducing $E$ if necessary, we can assume that there exists a constant~$c_{1}=c_{1}(c_{0})>0$ such that $E\subset c_{1}^{-1}I(z)$ and $m(E) \ge c_{1} m(I(z))$. Let $I_{0}=c_{1}^{-1}I(z)$ and consider its dyadic decomposition~$\mathcal{D}(I_{0})$. Fix $0<\gamma<1$ with  $1-\gamma<\varepsilon$. Let $\mathcal{G}_{0}=\{I_{k}: k\ge 1\}$ be the collection of maximal dyadic arcs in~$\mathcal{D}(I_{0})$ 
such that
\begin{equation}\label{eq3.14}
|f^{N}(z_{k})| \geq \gamma.
\end{equation}
Here $z_{k}=z(I_{k})$, $k\ge 1$. By maximality, the arcs in the collection $\mathcal{G}_{0}$ are pairwise disjoint. Note also that by maximality and Schwarz's Lemma, there exists a universal constant~$\tau >0$ such that
\begin{equation}\label{eq3.15}
\tau (1-	\gamma)\le 1-|f^{N}(z_{k})|\le 1-\gamma,\quad k\ge 1.
\end{equation}
This is the statement in (a). Since $f^{N}$ is inner, the union of $\{I_{k}:k\ge 1\}$ covers almost every point of~$I_{0}$. Consider the subcollection~$\mathcal{G}_{1} \subset\mathcal{G}_{0}$ of those dyadic arcs~$I_{k}\in\mathcal{G}_{0}$ such that
\begin{equation}\label{eq3.16}
m(E\cap I_{k})\ge c_{2} m(I_{k}).
\end{equation}
Here $c_{2}>0$ is a small constant to be fixed later. We will show that the subcollection~$\mathcal{G}_{1} $ satisfies the conditions (b) and (c) in the statement. Let $\mathcal{L}$ be the collection of indices~$k$ such that $I_{k}\in\mathcal{G}_{1}$. Note that
$$
c_{1}m(I(z))\le m(E)\le \sum_{k\notin\mathcal{L}}  m(E\cap I_{k}) +\sum_{k\in\mathcal{L}}m(I_{k}).
$$
If $k\notin \mathcal{L}$, we have $m(E\cap I_{k})\le c_{2}m(I_{k})$. Hence the first sum is bounded by~$c_{2}c_{1}^{-1} m(I(z))$. Pick $c_{2}>0$ sufficiently small so that $c_{1}-c_{2}c_{1}^{-1}>c_{1}/2$ to deduce
\begin{equation}\label{eq3.17}
\sum_{k\in \mathcal{L}} m( I_{k} ) \ge \frac{c_{1}}{2} m(I(z)).
\end{equation}
Choosing $0<c<c_1 / 2$, the estimate~\eqref{eq3.17} gives the statement in (b). Note that $|f^{N}(z)|\le |f^{M}(z)|<\varepsilon$ and $|f^{N}(z_{k})|\ge \gamma$,  for any positive integer $k$. By Schwarz's Lemma $\rho(z_{k},z)\ge \rho(\gamma,\varepsilon)$. Since $0<1-\gamma<\varepsilon$, taking $\varepsilon>0$ sufficiently small we deduce that $\rho (\gamma,\varepsilon)$ is close to~$1$ and then $m(I_{k}) \le c_3(\varepsilon) m(I_{0})$ for any $k \geq 1$, where $c_3 (\varepsilon)$ is tends to~$0$ as $\varepsilon$ tends to~$0$. If $c>0$ and  $\varepsilon >0$ are taken sufficiently small we have $c^{-1} I_k \subset c^{-1} I(z)$ for any $k\geq 1$. 

By Theorem~\ref{theo1.2}, there exists a constant~$c_4 = c_4 (f)>0$ such that for any~$k\ge 1$ we have
$$
\frac{1}{m(I_{k})} \int_{I_{k}}\left| F(\xi)- F_{I_k} \right|^{2}\,dm(\xi)\le c_4 \sum^{N}_{n=M}|a_{n}|^{2}(1-|f^{n}(z_{k})|^{2}).
$$
By part~(a) of Lemma~\ref{lem2.1}, there exists a constant~$0<c_{5}=c_{5}(\gamma,f)<1$ such that $1-|f^{n}(z_{k})|\le c_{5}^{N-n} (1-|f^{N}(z_{k})|)$, $n\le N$. Since $1-|f^{N}(z_{k})|\le 1-\gamma$, we obtain
$$
\frac{1}{m(I_{k})} \int_{I_{k}}\left|F(\xi)- F_{I_k}\right|^{2}\,dm(\xi)\le
2c_4 (1-\gamma)\sum^{N}_{n=M}|a_{n}|^{2}c_{5}^{N-n}.
$$
Recall that $m(E\cap I_{k})\ge c_{2} m(I_{k})$ if $k\in\mathcal{L}$. Then if $1-\gamma$ is sufficiently small, by Lemma~\ref{lem3.4}, the estimate~\eqref{eq3.13} implies that
$$
\frac{1}{m(I_{k})} \int_{I_{k}}\operatorname{Re}F\,dm \ge \frac{c_{0}}{2}\left(\sum^{N}_{n=M}|a_{n}|^{2}\right)^{1/2}.
$$
Choosing $0<c<c_0 / 2$, last estimate gives the statement in (c).This finishes the proof. 
\end{proof}

% \begin{lemma}\label{lem3.6}
% Let $F\in\operatorname{BMO}(\partial\mathbb{D})$ and $0<c<1$. Let $I\subset\partial\mathbb{D}$ be an % arc such that
%$$
% \frac{1}{m(I)}\int_{I}|F-F_{I}|^{2}\,dm\le c\|F\|^{2}_{\operatorname{BMO}(\partial\mathbb{D})}.
%$$
%Then there exists a constant $K_{0}=K_{0}(c)>0$ such that the following statement holds. Fix $K>K_{0}$ %and let $\{J_{k}\}$ be a family of pairwise disjoint arcs contained in~$I$ such that %$|F_{J_{k}}-F_{I}|\ge c^{1/2}\|F\|_{\operatorname{BMO}(\partial\mathbb{D})}$. Then
%$$
%\sum_{k}m(J_{k})\le \frac{m(I)}{K^{2}}.
%$$
%\end{lemma}
%
%\begin{proof}
%Note that for any $k=1,2,\dotsc$, we have
%$$
%K^{2}c\|F\|^{2}_{\operatorname{BMO}(\partial\mathbb{D})}\le |F_{J_{k}}-F_{I}|^{2}\le %\frac{1}{m(J_{k})}\int_{J_{k}} |F-F_{I}|^{2}\,dm.
%$$
%Adding over $k$, we obtain
%$$
%K^{2}c\|F\|^{2}_{\operatorname{BMO}(\partial\mathbb{D})}\sum_{k} m(J_{k})\le %\int_{I}|F-F_{I}|^{2}\,dm\le  c\|F\|^{2}_{\operatorname{BMO}(\partial\mathbb{D})}.
%$$
% \end{proof}
Our next auxiliary result is the building block of the main construction in the proof of Theorem~\ref{theo1.5}.

\begin{lemma}\label{lem3.7}
Let $f$ be an inner function with $f(0)=0$ which is not a rotation. Then there exist constants~$0<\varepsilon=\varepsilon(f)<1$ and $0<c=c(f)<1$ such that the following statement holds. Let $0<\gamma<1$ be a constant with $1-\gamma\le \varepsilon$, let $\{a_{n}\}$ be a sequence of complex numbers, let $M_{1}<N_{1}<M_{2}<N_{2}$ be positive integers and
$$
F_{i}=\sum^{N_{i}}_{n=M_{i}}a_{n}f^{n},\quad i=1,2.
$$
Let $I\subset \partial\mathbb{D}$ be an arc such that $c(1-\gamma)< 1-|f^{N_{1}}(z(I))|\le 1-\gamma$ and $|f^{M_{2}}(z(I))|\le \varepsilon$. Then there exists an arc~$J$ with $c^{-1}J\subset c^{-1}I$, such that
\begin{gather}
c (1-\gamma)\le 1-|f^{N_{2}}(z(J))|\le 1-\gamma,\label{eq3.18}\\
\frac{1}{|J|} \int_{J}\operatorname{Re}F_{2}\,dm \ge c\left(\sum^{N_{2}}_{n=M_{2}}|a_{n}|^{2}\right)^{1/2},\label{eq3.19}\\
\frac{1}{|J|} \int_{J}\operatorname{Re}F_{1}\,dm \ge \frac{1}{|I|}\int_{I}\operatorname{Re}F_{1}\,dm-
c^{-1}(1-\gamma)^{1/2}
\left(\sum^{N_{1}}_{n=M_{1}}|a_{n}|^{2}t^{N_{1}-n}\right)^{1/2},\label{eq3.20}
\end{gather}
where $t=(1+\gamma)^{-1} (1+|f'(0)|\gamma)<1$.
\end{lemma}

\begin{proof}
Let $\varepsilon_0>0$ and $c_0 >0$ be the constants appearing in Lemma~\ref{lem3.5}. Let $0 < \varepsilon < \varepsilon_0$ be a constant to be fixed later. Let $0<\gamma<1$ with $1-\gamma<\varepsilon$. Apply Lemma~\ref{lem3.5} to the function $F_2$ to find arcs~$J_{k}$ with $c_{0}^{-1}J_{k}\subset c_{0}^{-1}I$, such that $\tau (1-\gamma)\le 1-|f^{N_{2}}(z(J_{k}))|\le 1-\gamma$, $\sum m(J_{k})\ge c_{0} m(I)$ and
$$
\frac{1}{|J_{k}|}\int_{J_{k}}\operatorname{Re} F_{2}\,dm\ge c_{0}\left( \sum^{N_{2}}_{n=M_{2}}|a_{n}|^{2}\right)^{1/2}.
$$
Here $\tau>0$ is the universal constant given by part (a) of Lemma~\ref{lem3.5}. The constant $c$ will satisfy $0<c<c_0$ and the arc $J$ will be one of the arcs of the family $\{J_k \}$. Hence estimates \eqref{eq3.18} and \eqref{eq3.19} will follow from the construction. Theorem~\ref{theo1.2} gives that there exists a constant $c_1=c_1(f)>0$ such that 
\begin{equation}\label{3.100}
\frac{1}{c_{0}^{-1}m(I)}\int_{c_{0}^{-1}I}|\mathrm{Re}\,F_{1}- \mathrm{Re}\,F_{1} (z(c_{0}^{-1} I))|^{2}\,dm\le c_1 \sum^{N_{1}}_{n=M_{1}}|a_{n}|^{2} (1-|f^{n}(z(c_{0}^{-1} I))|^{2}).
\end{equation}
Note that 
\begin{equation*}
\begin{split}
| (\operatorname{Re}F_{1})_{I} -  \mathrm{Re}\,F_{1} (z(c_{0}^{-1} I)) | &\leq \frac{1}{m(I)} \int_I |\mathrm{Re}\,F_{1}- \mathrm{Re}\,F_{1} (z(c_{0}^{-1} I))| dm  \\
&\leq \left(  \frac{1}{ m(I)}  \int_{ I} |\mathrm{Re}\,F_{1}- \mathrm{Re}\,F_{1} (z(c_{0}^{-1} I))|^2 dm 
\right)^{1/2}
\end{split}
\end{equation*}
Applying \eqref{3.100} we find a constant $c_2 = c_2 (f, c_0) >0$ such that
\begin{equation}\label{3.101}
  \frac{1}{m(I)}\int_{c_{0}^{-1}I}|\mathrm{Re}\,F_{1}-  (\operatorname{Re}F_{1})_{I} |^{2}\,dm\le c_2 \sum^{N_{1}}_{n=M_{1}}|a_{n}|^{2} (1-|f^{n}(z(c_{0}^{-1} I))|^{2}).  
\end{equation}
By Schwarz's Lemma, there exists a constant $c_3 = c_3 (c_0) >0$ such that  $1-|f^{n}(z(c_{0}^{-1} I ))|^{2} \leq c_3 (1-|f^{n}(z(I))|^{2})$ for any $n \geq 1$. Since $1-|f^{N_{1}}(z(I))|\le 1-\gamma$, part~(a) of Lemma~\ref{lem2.1} provides a constant~$0<c_4<1$ such that $1-|f^{n}(z(I))|\le c_4^{N_{1}-n}(1-\gamma)$, $n\le N_{1}$. Actually, according to Lemma~\ref{lem2.1} one can take $c_4=(1+\gamma)^{-1}(1+|f'(0)|\gamma)$. Applying \eqref{3.101} we obtain a constant $c_5 = c_5 (c_0 , f) >0$ such that
\begin{equation}\label{3.102}
  \frac{1}{m(I)}\int_{c_{0}^{-1}I}|\mathrm{Re}\,F_{1}-  (\operatorname{Re}F_{1})_{I} |^{2}\,dm\le c_5 (1- \gamma) \sum^{N_{1}}_{n=M_{1}}|a_{n}|^{2}  c_4^{N_{1}-n} .  
\end{equation}
Let $K>0$ be a large constant to be fixed later. Let $\mathcal{G}$ be the subcollection of those arcs~$J_{k}$ such that
$$
(\operatorname{Re}F_{1})_{J_{k}}\le (\operatorname{Re}F_{1})_{I}-K(1-\gamma)^{1/2} \left(\sum^{N_{1}}_{n=M_{1}}|a_{n}|^{2} c_4^{N_{1}-n}\right)^{1/2}.
$$
Note that for any $J_k \in \mathcal{G}$ we have 
$$
\frac{1}{m(J_k)} \int_{J_k} | \operatorname{Re}F_{1}  - (\operatorname{Re}F_{1})_{I}| dm \geq |(\operatorname{Re}F_{1})_{J_k} - (\operatorname{Re}F_{1})_{I}| \geq K(1-\gamma)^{1/2} \left(\sum^{N_{1}}_{n=M_{1}}|a_{n}|^{2} c_4^{N_{1}-n}\right)^{1/2} . 
$$
Adding over $k$, we obtain
$$
\sum_{J_{k}\in\mathcal{G}} \int_{J_k} | \operatorname{Re}F_{1}  - (\operatorname{Re}F_{1})_{I}| dm \geq  K(1-\gamma)^{1/2} \left(\sum^{N_{1}}_{n=M_{1}}|a_{n}|^{2} c_4^{N_{1}-n}\right)^{1/2} \sum_{J_{k}\in\mathcal{G}}m(J_{k}) . 
$$
Applying \eqref{3.102} we deduce 
$$
\sum_{J_{k}\in\mathcal{G}}m(J_{k})\le \frac{c_5^{1/2}}{K} m(I).
$$
Since $\sum m(J_{k})\ge c_{0}m(I)$, taking $K>0$ large enough so that $c_5^{1/2} K^{-1} < c_0 / 2$, we deduce 
$$
\sum_{J_{k}\notin\mathcal{G}}m(J_{k})\ge \frac{c_{0}}{2}m(I)
$$
 and one can take as $J$ in the statement, any of the arcs~$J_{k}\notin\mathcal{G}$.
\end{proof}

We are now ready to prove Theorem\ref{theo1.5}. 

\begin{proof}[Proof of Theorem~\ref{theo1.5}]

Since $\sum a_{n}f^{n}$ converges at a set of positive measure, Theorem~\ref{theo1.1} gives that $\sum |a_{n}|^{2}<\infty$. Hence the function~$F=\sum a_{n}f^{n}$ is in~$\operatorname{BMOA}$. The plan of the proof is to find constants~$A,B>0$, a sequence of arcs~$\{I_{k}\}$ contained in~$I$ and a sequence of positive integers~$\{M_{k}\}$ tending to infinity such that
\begin{equation}\label{eq3.22}
\frac{1}{|I_{k}|} \int_{I_{k}}\operatorname{Re}F\,dm\ge A\sum^{M_{k}}_{n=1}|a_{n}|-B.
\end{equation}
It is clear that \eqref{eq3.22} finishes the proof. The construction of the arcs~$\{I_{k}\}$ and the sequence~$\{M_{k}\}$ is performed inductively. We start using an idea of Weiss (\cite{We}). Let $T'<T$ be two (large) positive integers to be fixed later such that $T/T'$ is integer. Split the sum~$F=\sum a_{n}f^{n}$ into blocks of length~$T$, that is, $F=\sum\limits_{k\ge 0}G_{k}$, where
$$
G_{k}=\sum^{T-1}_{n=0}a_{kT+n}f^{kT+n},\quad k=0,1,2,\dotsc
$$
Next split $G_{2k}$ into successive blocks of length~$T'$ and pick the subblock for which the sum of the modulus of the coefficients is the least. In other words, pick $\mathcal{G}_{k}$ a subset of $T'$ consecutive integers in the interval~$[2kT,(2k+1)T)$ such that
$$
\sum_{n\in \mathcal{G}_{k}}|a_{n}|\le \sum_{\ell \in \mathcal{G}}|a_{\ell}|
$$
for any subset~$\mathcal{G}$ of $T'$ consecutive integers in~$[2kT,(2k+1)T)$. Since the number of disjoint subblocks of length $T'$ in~$[2kT,(2k+1)T)$ is $T/T'$ we deduce
\begin{equation}\label{eq3.23}
\sum_{n\in \mathcal{G}_{k}}|a_{n}| \le \frac{T'}{T}\sum^{T-1}_{n=0}|a_{2kT+n}|.
\end{equation}
Each set of indices~$\mathcal{G}_{k}$, $k\ge 1$, has $T'$~elements and the corresponding block
$$
S_{k}=\sum_{n\in\mathcal{G}_{k}}a_{n}f^{n}
$$
will be called a short block. The long blocks are defined as the blocks between two short blocks. More concretely
$$
L_{1}=\sum^{N_{1}}_{n=0}a_{n}f^{n},
$$
where $N_{1}+1$ is the first index in $\mathcal{G}_{1}$ and
$$
L_{k}=\sum^{N_{k}}_{n=M_{k}}a_{n}f^{n},\quad k>1,
$$
where $M_{k}-1$ is the largest index in~$\mathcal{G}_{k-1}$ and $N_{k}+1$ is the smallest index in~$\mathcal{G}_{k}$. Note that each short block has $T'$~terms while the number of terms in a long block is between $T$ and $3T$. Note also that if $\mathcal{L}_{k}=\{n\in\mathbb{Z}: M_{k}\le n\le N_{k}\}$ denotes the set of indices appearing in the long block~$L_{k}$, then \eqref{eq3.23} implies
\begin{equation}\label{eq3.24}
\sum^{L}_{k=1}\sum_{n\in \mathcal{L}_{k}}|a_{n}|\ge \left(1-\frac{T'}{T}\right) \sum^{N_{L}}_{n=1}|a_{n}|,\quad L=1,2,\dotsc
\end{equation}
The idea is that \eqref{eq3.24} will imply that the short blocks are irrelevant and the construction of the arcs~$\{I_{k}\}$ and the sequence of integers~$\{M_{k}\}$ verifying \eqref{eq3.22} will depend on the long blocks. Moreover, the fact that two long blocks are separated by a short block will provide a sort of independence between the long blocks.

Let $s_{k}$ denote the $\ell^{2}$-norm of the coefficients in the long block~$L_{k}$, that is,
$$
s_{k}^{2}=\sum_{n\in \mathcal{L}_{k}} |a_{n}|^{2},\quad k=1,2,\dotsc
$$ 
Let $\varepsilon=\varepsilon(f)>0$ and $c=c(f)>0$ be the constants given by Lemma~\ref{lem3.7}. Let $0<\gamma=\gamma(f)<1$ be a constant to be fixed later satisfying $1-\gamma<\varepsilon$. Let $D(0,R)$ denote the disc centered at the origin of radius $R$. By the Denjoy--Wolff Theorem, $f^{n}$ tends to~$0$ uniformly on compact sets of~$\mathbb{D}$. The integer~$T'$ will be taken large enough so that
\begin{equation}\label{eq3.25}
% (1-c_{0}(1-\gamma))^{T'}<\varepsilon_{0}.
f^{n} ( D(0,  1- c (1- \gamma) )) \subset D(0, \varepsilon), \quad n \geq T' .  
\end{equation}
Fix an arc $I^{*}\subset I$ such that $c^{-1}I^{*}\subset I$. The construction starts with the first long block~$L_{k_{0}}$ such that $|f^{M_{k_{0}}}(z(I^{*}))|<\varepsilon$. Without loss of generality we can assume that the constant $c$ is smaller than the constants appearing in Lemma~\ref{lem3.5}. Apply Lemma~\ref{lem3.5} to find an arc~$I_{k_{0}} $ with $c^{-1}I_{k_{0}}\subset c^{-1}I^{*}\subset I$ such that
\begin{gather*}
c(1-\gamma)\le 1-|f^{N_{k_{0}}} (z(I_{k_{0}}))|\le 1-\gamma,\\*[5pt]
\frac{1}{m(I_{k_{0}})}\int_{I_{k_{0}}}\operatorname{Re} L_{k_{0}}\,dm\ge c s_{k_{0}}.
\end{gather*}
Assume by induction that we have constructed arcs~$I_{k_{0}},I_{k_{0}+1},\dotsc, I_{k}$, $k \geq k_{0}$, with $c^{-1}I_{j}\subset c^{-1}I_{j-1}$, $j=k_{0}+1,\dotsc, k$, such that the following two conditions hold
\begin{gather}
c(1-\gamma)\le 1-|f^{N_{k}}(z(I_{k}))|\le 1-\gamma,\label{eq3.26}\\*[5pt]
\frac{1}{m(I_{k})}\int_{I_{k}}\operatorname{Re}\left(\sum^{k}_{j=k_{0}}L_{j}\right)\,dm\ge c\sum^{k}_{j=k_{0}}s_{j}
-c^{-1}(1-\gamma)^{1/2}\sum^{k-1}_{\ell=k_{0}}\left(\sum^{\ell}_{j=k_{0}}t^{N_{\ell}-N_{j}}s_{j}^{2}\right)^{1/2}.\label{eq3.27}
\end{gather}
Here $0<t<1$ is the constant appearing in Lemma~\ref{lem3.7} and when $k=k_0$, we replace the right hand side term of \eqref{eq3.27} by $c s_{k_0}$. Recall that two different long blocks are separated by a short block which has $T'$ terms. Hence the estimates~\eqref{eq3.25} and \eqref{eq3.26} give that $|f^{M_{k+1}}(z(I_{k}))|\le \varepsilon$. Apply Lemma~\ref{lem3.7} to $F_{1}=\sum\limits_{j=k_{0}}^{k}L_{j}$ and $F_{2}=L_{k+1}$, to find an arc~$I_{k+1}$  with $c^{-1}I_{k+1}\subset c^{-1}I_{k}$ such that
\begin{gather}
c (1-\gamma)\le 1-|f^{N_{k+1}}(z(I_{k+1}))|\le 1-\gamma,\nonumber\\*[5pt]
\frac{1}{m(I_{k+1})}\int_{I_{k+1}}\operatorname{Re}L_{k+1}\,dm\ge c s_{k+1},\label{eq3.28}\\*[5pt]
\frac{1}{m(I_{k+1})}\int_{I_{k+1}}\operatorname{Re}\sum^{k}_{j=k_{0}}L_{j}\,dm\ge \frac{1}{m(I_{k})}\int_{I_{k}}
\operatorname{Re}\sum^{k}_{j=k_0}L_{j}\,dm-
c^{-1}(1-\gamma)^{1/2}\left(\sum^{k}_{j=k_{0}}t^{N_{k}-N_{j}}s_{j}^{2}\right)^{1/2}.\label{eq3.29}
\end{gather}
Then \eqref{eq3.29} and the induction assumption~\eqref{eq3.27} give
$$
\frac{1}{m(I_{k+1})}\int_{I_{k+1}}\operatorname{Re}\left(\sum^{k}_{j=k_{0}}L_{j}\right)\,dm\ge c \sum^{k}_{j=k_{0}}s_{j}-c^{-1}(1-\gamma)^{1/2}\sum^{k}_{\ell=k_{0}}\left(\sum^{\ell}_{j=k_{0}}t^{N_{\ell}-N_{j}}s_{j}^{2}\right)^{1/2}.
$$
Applying \eqref{eq3.28} we deduce
$$
\frac{1}{m(I_{k+1})}\int_{I_{k+1}}\operatorname{Re}\left(\sum^{k+1}_{j=k_{0}}L_{j}\right)\,dm\ge c \sum^{k+1}_{j=k_{0}}s_{j}- c^{-1}(1-\gamma)^{1/2}\sum^{k}_{\ell=k_{0}}\left(\sum^{\ell}_{j=k_{0}}t^{N_{\ell}-N_{j}}s_{j}^{2}\right)^{1/2}.
$$
This concludes the inductive step.

Note that for any $k>k_{0}$, we have
$$
\sum^{k-1}_{\ell=k_{0}}\left(\sum^{\ell}_{j=k_{0}}t^{N_{\ell}-N_{j}}s_{j}^{2}\right)^{1/2}\le \sum^{k-1}_{\ell=k_{0}}\sum^{\ell}_{j=k_{0}}t^{(N_{\ell}-N_{j})/2}s_{j}=
\sum^{k-1}_{j=k_{0}}s_{j }\sum^{k-1}_{\ell=j}t^{(N_{\ell}-N_{j})/2}\le (1-t^{1/2})^{-1}
\sum^{k-1}_{j=k_{0}}s_{j}.
$$
Pick $0< \gamma< 1$ sufficiently close to~$1$ so that $c^{-1}(1-\gamma)^{1/2}(1-t^{1/2})^{-1}\le c/2$. The estimate~\eqref{eq3.27} gives
\begin{equation}\label{eq3.30}
\frac{1}{m(I_{k})}\int_{I_{k}}\operatorname{Re}\left(\sum^{k}_{j=k_{0}}L_{j}\right)\,dm\ge \frac{c}{2}\sum^{k}_{j=k_{0}}s_{j}.
\end{equation}
Recall that each long block has at most $3T$~terms. Hence Cauchy-Schwarz's inequality gives 
$$
\sum_{n\in \mathcal{L}_{j}}|a_{n}|\le s_{j}(3T)^{1/2},\quad j=1,2,\dotsc
$$
Then apply \eqref{eq3.30} to deduce
\begin{equation}\label{eq3.31}
\frac{1}{m(I_{k})} \int_{I_{k}}\operatorname{Re} \sum^{k}_{j=k_{0}} (L_{j}+S_{j})\,dm\ge \frac{c}{2} (3T)^{-1/2}\sum^{k}_{j=k_{0}}\sum_{n\in\mathcal{L}_{j}}|a_{n}|-\sum^{k}_{j=k_{0}}\sum_{n\in\mathcal{G}_{j}}|a_{n}|,\quad k\ge k_{0}.
\end{equation}
Let $\mathcal{A}_{k}$ be the set of indices appearing in $\sum\limits_{j=k_{0}}^{k}(L_{j}+S_{j})$, that  is, $\mathcal{A}_{k}=\bigcup\limits_{j=k_{0}}^{k}(\mathcal{L}_{j}\cup\mathcal{G}_{j})$, $k\ge k_{0}$. Applying \eqref{eq3.31}, \eqref{eq3.23} and \eqref{eq3.24} we obtain
$$
\frac{1}{m(I_{k})}\int_{I_{k}}\operatorname{Re} \sum^{k}_{j=k_{0}} (L_{j}+S_{j})\,dm \ge \frac{c}{2} (3T)^{-1/2}\left(1-\frac{2T'}{T}\right)\sum_{n\in\mathcal{A}_{k}}|a_{n}|-\frac{2T'}{T}\sum_{n\in\mathcal{A}_{k}}|a_{n}|.
$$
We now choose $T$ so that $T'T^{-1/2}$ is sufficiently small so that
$$
\frac{c}{2} (3T)^{-1/2}\left(1-\frac{2T'}{T}\right)-\frac{2T'}{T}\ge \frac{c}{4}T^{-1/2}.
$$
Then
\begin{equation}\label{eq3.32}
\frac{1}{m(I_{k})}\int_{I_{k}}\operatorname{Re} \sum^{k}_{j=k_{0}}(L_{j}+S_{j})\,dm\ge \frac{c}{4} T^{-1/2}\sum_{n\in\mathcal{A}_{k}}|a_{n}|.
\end{equation}
As observed previously, $|f^{M_{k+1}}(z(I_{k}))|\le \varepsilon$. By Corollary~\ref{coro2.2} there exists a constant~$0<c_{1}<1$ such that $|f^{n}(z(I_{k}))|\le c_{1}^{n-M_{k+1}}$, $n\ge M_{k+1}$. Then
$$
\sum_{n\ge M_{k+1}}|a_{n}| |f^{n}(z(I_{k}))|\le (1-c_{1}^2)^{-1/2}\left( \sum_{n\ge M_{k+1}}|a_{n}|^{2}\right)^{1/2}.
$$
Since there exists a constant~$c_2 = c_2(f)>0$ such that
$$
\left|\frac{1}{m(I_{k})}\int_{I_{k}}\sum_{n\ge M_{k+1}}a_{n}f^{n}\,dm-\sum_{n\ge M_{k+1}}a_{n}f^{n}(z(I_{k}))\right|\le c_2 \left(\sum_{n\ge M_{k+1}}|a_{n}|^{2}\right)^{1/2},
$$
the estimate \eqref{eq3.32} gives
$$
\frac{1}{m(I_{k})}\int_{I_{k}}\operatorname{Re}\sum^{\infty}_{n=1}a_{n}f^{n}\,dm\ge \frac{c}{4} T^{-1/2}\sum_{n\in\mathcal{A}_{k}}|a_{n}|-(c_2 +(1-c_{1}^2)^{-1/2})\left(\sum_{n\ge 1}|a_{n}|^{2}\right)^{1/2}.
$$
This gives \eqref{eq3.22} and finishes the proof.
\end{proof}

\section{Other function spaces}\label{sec4}

Let $A(\mathbb{D})$ denote the disc algebra, that is, the space of analytic functions in~$\mathbb{D}$ which extend continuously to $\overline{\mathbb{D}}=\{z\in\mathbb{C}: |z|\le 1\}$. The only inner functions which belong to~$A(\mathbb{D})$ are the finite Blaschke products. See ~\cite[p.~72]{Ga}. Our next result is an easy consequence of Theorem~\ref{theo1.5}.

\begin{theorem}\label{theo4.1}
Let $f$ be an inner function with $f(0)=0$ which is not a rotation and let $\{a_{n}\}$ be a sequence of complex numbers not identically zero. Then $\sum\limits_{n=1}^{\infty} a_{n}f^{n}\in A(\mathbb{D})$ if and only if $f$ is a finite Blaschke product and $\sum\limits_{n=1}^{\infty} |a_{n}|<\infty$.
\end{theorem}

\begin{proof}
Write $F=\sum\limits_{n=1}^{\infty} a_{n}f^{n}$. The sufficiency follows from the uniform convergence of the partial sums~$\sum\limits_{n=1}^{N}a_{n}f^{n}$. We now prove the necessity. Assume $F\in A(\mathbb{D})$. Since $\sup \{|F(\xi)|: \xi\in \partial \mathbb{D}\}<\infty$, Theorem~\ref{theo1.5} gives that $\sum\limits_{n=1}^{\infty} |a_{n}|<\infty$. We now show that $f$ is a finite Blaschke product arguing by contradiction. Assume that $f$ does not extend analytically at any neighbourhood of the point~$\xi\in\partial \mathbb{D}$. By Frostman's Theorem (see~\cite[p.~77]{Ga}) there exists  a set~$E\subset\mathbb{D}$ of logarithmic capacity zero such that for any $\alpha\in \mathbb{D}\backslash E$ one can find a sequence~$\{z_{k} \} = \{ z_{k}(\alpha) \}$, of points in~$\mathbb{D}$ converging to~$\xi$, such that $f(z_{k})=\alpha$ for any~$k$.
Then
$$
F(z_{k})=\sum^{\infty}_{n=1}a_{n}f^{n-1}(\alpha),\quad k=1,2,\dotsc
$$
Here $f^0 (\alpha) = \alpha , \alpha \in \mathbb{D}$. Since $\{z_{k}\}$ converges to~$\xi$ and $F \in A(\mathbb{D})$, for any $\alpha\in\mathbb{D}\backslash E$ we have
$$
\sum_{n=1}^{\infty} a_{n}f^{n-1}(\alpha)=F(\xi).
$$
Hence the function $\sum\limits^{\infty}_{n=1}a_{n}f^{n-1}$ is constant in $\mathbb{D}$. Since it  vanishes at the origin, we deduce\linebreak $\sum\limits_{n=1}^{\infty} a_{n}f^{n-1}\equiv 0$. Hence $\sum\limits_{n=1}^{\infty} a_{n}f^{n}\equiv 0$. By Lemma \ref{lem2.3}, $a_n=0$ for any $n \geq 1$, which gives the contradiction.
\end{proof}

Let $\mathcal{D}$ denote the Dirichlet space of analytic functions~$f$ in~$\mathbb{D}$ such that
$$
\|f\|^{2}_{\mathcal{D}}=\int_{\mathbb{D}}|f'(z)|^{2}\,dA(z)<\infty,
$$
where $dA(z)$ denotes the area measure. Note that $\|f\|^{2}_{\mathcal{D}}$ can be understood as the area of the image domain~$f(\mathbb{D})$, counting multiplicities. Hence the only inner functions which belong to the Dirichlet space are finite Blaschke products.  Our next result is a description of linear combinations of iterates of an inner function which belong to the Dirichlet space.

\begin{theorem}\label{theo4.2}
Let $f$ be a finite Blaschke product with $N>1$~zeros and with $f(0)=0$. Let $\{a_{n}\}$ be a sequence of complex numbers with $\sum\limits_{n=1}^{\infty}|a_{n}|^{2}<\infty$ and $F=\sum\limits_{n=1}^{\infty}a_{n}f^{n}$. Then $F \in \mathcal{D}$ if and only if $\sum\limits^{\infty}_{n=1}|a_{n}|^{2}N^{n} < \infty$. Moreover there exists a universal constant $C=C(f)>0$ such that
$$
C^{-1}\sum^{\infty}_{n=1}|a_{n}|^{2}N^{n}\le \int_{\mathbb{D}}|F'(z)|^{2}\,dA(z)\le C\sum^{\infty}_{n=1}|a_{n}|^{2}N^{n}.
$$
\end{theorem}

\begin{proof}
Note that
$$
\int_{\mathbb{D}}|F'(z)|^{2}\,dA(z)=\sum^{\infty}_{n=1}|a_{n}|^{2}b_{n,n}+2\operatorname{Re} \sum^{\infty}_{n=1}\overline{a}_{n}\sum_{k>n}a_{k}b_{k,n},
$$
where
$$
b_{k,n}=\int_{\mathbb{D}}\overline{(f^{n})'(z)} (f^{k})'(z)\,dA(z),\quad k,n=1,2,\dotsc
$$
Since $f^{n}$ is a Blaschke product with $N^{n}$~zeros, the area counting multiplicities, of~$f^{n}(\mathbb{D})$ is~$\pi N^{n}$. Then
$$
\int_{\mathbb{D}}|(f^{n})'(z)|^{2}\,dA(z)=\pi N^{n}.
$$
On the other hand, the change of variables formula with multiplicities (see~\cite[p.~122]{EG}) gives
$$
b_{k,n}=\int_{\mathbb{D}}|(f^{n})'(z)|^{2} (f^{k-n})' (f^{n}(z))\,dA(z)=
\int_{\mathbb{D}} N^{n}(f^{k-n})'(w)\,dA(w),\quad k>n.
$$
The mean value property gives $b_{k,n}=\pi N^{n} f'(0)^{k-n}$, $k>n$. Hence
$$
\frac{1}{\pi}\int_{\mathbb{D}} |F'(z)|^{2}\,dA(z)=\sum^{\infty}_{n=1}|a_{n}|^{2}N^{n}+2\operatorname{Re} \sum^{\infty}_{n=1} \overline{a}_{n}N^{n}\sum_{k>n}a_{k} f'(0)^{k-n} =\vec{a}^{\,t} T \vec{a},
$$
where $\vec{a}$ denotes the vector~$(a_{n} N^{n/2})_{n\ge 1}$ and $T$ is the Toeplitz matrix whose entries are
$$
t_{n,k}=(f'(0) N^{-1/2})^{k-n}, \quad k \geq n;  \qquad t_{n,k}=(\overline{f'(0)} N^{-1/2})^{n-k}, \quad n \geq k.
$$
Consider the symbol
$$
t(\xi)=\sum^{\infty}_{n=-\infty} t_{n,0} \xi^{n}= \frac{1-|f'(0)|^{2}N^{-1}} {|1- f'(0) N^{-1/2}\xi|^{2}},\quad \xi\in\partial\mathbb{D}.
$$
It is well known that $T$ diagonalizes and its eigenvalues are between the minimum and the maximum of~$t$. See~\cite{BG}. Since $C(f,N)^{-1}\le t(\xi)\le C(f,N)$, a.e.\ $\xi\in\partial\mathbb{D}$, where
$$
C(f,N)=\frac{1+|f'(0)|N^{-1/2}}{1-|f'(0)|N^{-1/2}},\quad \text{a.e.\ } \xi\in\partial\mathbb{D},
$$
we deduce that
$$
C(f,N)^{-1}\sum^{\infty}_{n=1}|a_{n}|^{2} N^{n}\le \frac{1}{\pi}\int_{\mathbb{D}} |F'(z)|^{2}\,dA(z) \le C(f,N) \sum^{\infty}_{n=1} |a_{n}|^{2} N^{n}.
$$
\end{proof}

Let $\mathcal{B}$ denote the Bloch space of analytic functions~$f$ in~$\mathbb{D}$ such that
$$
\|f\|_{\mathcal{B}}=\sup_{z\in\mathbb{D}} (1-|z|^{2}) |f'(z)|<\infty.
$$
It is well known that a lacunary series is in the Bloch space if and only if its coefficients are uniformly bounded. See~\cite{ACP}. In our setting this condition is still sufficient but we will see that it is not necessary.

\begin{theorem}\label{theo4.3}
Let $f$ be an inner function with $f(0)=0$ which is not a rotation and let $\{a_{n}\}$ be a bounded sequence of complex numbers. Then $F=\sum\limits_{n=1}^{\infty} a_{n}f^{n}\in\mathcal{B}$ and there exists a constant $C=C(f)>0$ such that $\|F\|_{\mathcal{B}}\le C \sup\limits_{n \geq 1}|a_{n}|$.
\end{theorem}

\begin{proof}
Let $0<r_{0}=r_{0}(f)<1$  and $0<c_{0}=c_{0}(f)<1$ be the constants given by Corollary~\ref{coro2.2}. Since $f^{n}$ tends to~$0$ uniformly on compacts of~$\mathbb{D}$, for any $z\in \mathbb{D}$ we can pick $N(z)$ to be the minimum positive  integer $n$ such that $|f^{n}(w)|\le r_{0}$ for any $w\in \mathbb{D}$ with $\rho(w,z)\le 1/2$. Corollary~\ref{coro2.2} gives
\begin{equation}\label{eq4.1}
\sup\{ |f^{n}(w)|: \rho(w,z)\le r_{0}\}\le c_{0}^{n-N(z)},\quad n\ge N(z), \quad z \in \mathbb{D}. 
\end{equation}
By Cauchy's estimate, there exist a universal constant $c_{1}>0$ such that
$$
(1-|z|)| (f^{n})'(z)|\le c_{1}c_{0}^{n-N(z)},\quad n\ge N(z), \quad z \in \mathbb{D}. 
$$
Hence
\begin{equation}\label{eq4.2}
(1-|z|)\sum_{n\ge N(z)}|(f^{n})'(z)|\le c_{1} (1-c_{0})^{-1}, \quad z \in \mathbb{D}. 
\end{equation}
Let $M= N(z) - 1$. Note that there exists a point $w$ with $\rho (w, z) \leq 1/2 $ such that $|f^M (w)| \geq r_0$. According to part~(a) of Lemma~\ref{lem2.1}, there exists a constant~$0<c_{2}<1$ such that $1-|f^{n}(w)|\le c_{2}^{M-n}$ for any $n <  N(z)$.  By Schwarz's Lemma there exists a constant $c_3 >0$ such that $1-|f^{n}(z)|\le c_3 c_{2}^{M-n}$ for any $n < N(z)$. Since $(1-|z|^{2})|(f^{n})'(z)|\le 1 - |f^{n}(z)|^{2}$, $z\in \mathbb{D}$, we deduce
\begin{equation}\label{eq4.3}
(1-|z|^{2})\sum_{n=1}^{M} |(f^{n})'(z)|\le c_3 (1-c_{2})^{-1}, \quad z \in \mathbb{D}. 
\end{equation}
The estimates \eqref{eq4.2} and \eqref{eq4.3} give that $\|F\|_{\mathcal{B}}\le c\sup\limits_{n}|a_{n}|$ with $c=c_{1}(1-c_{0})^{-1}+ c_3 (1-c_{2})^{-1}$.
\end{proof}

Next we will show that the converse estimate in Theorem~\ref{theo4.3} does not hold, that is, there exist an inner function~$f$ with $f(0)=0$ and unbounded sequence~$\{a_{n}\}$ of complex numbers such that $\sum a_{n}f^{n}\in\mathcal{B}$. We start with an auxiliary result on the hyperbolic derivative~$D_{h}f$ defined in \eqref{eq2.1}. Denote $f^{0}(z)=z, z \in \mathbb{D}$.

\begin{lemma}\label{lem4.4}
Let $f\in H^{\infty}(\mathbb{D})$, $\|f\|_{\infty}\le 1$. Then for any $n=1,2,\dotsc$, we have
$$
D_{h}(f^{n})(z)=\prod^{n-1}_{k=0}D_{h}(f) (f^{k}(z)),\quad z\in\mathbb{D}.
$$
\end{lemma}

\begin{proof}
Note that $(f^{n})'(z)=\prod\limits_{k=0}^{n-1} f'(f^{k}(z))$, $n\ge 1$. Then
$$
(1-|z|^{2})\frac{(f^{n})'(z)}{1-|f^{n}(z)|^{2}}=\prod^{n-1}_{k=0}
\frac{(1-|f^{k}(z)|^{2}f'(f^{k}(z))}{1-|f^{k+1}(z)|^{2}}.
$$
\end{proof}

Let $f\in H^{\infty}$ with $\|f\|_{\infty}\le 1$. Schwarz's Lemma gives that $D_{h}f(z)\le 1$, $z\in\mathbb{D}$. Given any positive gauge function $w\colon [0,1]\to (0,+\infty)$ satisfying a mild regularity condition, such that
$$
\int^{1}_{0}\frac{w^{2}(t)}{t}\,dt = 	\infty, 
$$
there exists an inner function~$f$ with $f(0)=0$ such that $D_{h}f(z)\le w(1-|z|)$, $z\in\mathbb{D}$. See \cite{AAN}  or \cite{Sm}. In particular, given any $0< \tau <1$, there exists an inner function~$f$ with $f(0)=0$ such that  $D_{h}f(z)\le \tau$. Let $\{a_{n}\}$ be any sequence of complex numbers such that
\begin{equation}\label{eq4.4}
\sum^{\infty}_{n=1}|a_{n}| \tau^{n}<\infty.
\end{equation}
We next show that the function~$F=\sum\limits^{\infty}_{n=1}a_{n}f^{n}$ is in the Bloch space. Note that
$$
(1-|z|^{2})|F'(z)| \leq \sum^{\infty}_{n=1} |a_{n} | |(f^{n})' (z) | (1-|z|^{2})=\sum^{\infty}_{n=1} |a_{n}|
(1-|f^{n}(z)|^{2})D_{h}(f^{n})(z) , \quad z \in \mathbb{D}. 
$$ 
Apply Lemma~\ref{lem4.4} to deduce $D_{h}(f^{n})(z)\le \tau^{n}$, $n\ge 1$, $z \in \mathbb{D}$. Then
$$
(1-|z|^{2})|F'(z)|\le \sum^{\infty}_{n=1}|a_{n}| \tau^{n}, \quad z \in \mathbb{D}.  
$$
Hence if condition~\eqref{eq4.4} holds, we have $F\in\mathcal{B}$.

We finish this section showing the following converse of part (b) of Corollary 1.3. 

\begin{corollary}\label{coro1.10}
Let $f$ be an inner function with $f(0)=0$ which is not a rotation and let $\{a_{n}\}$ be a non-identically zero sequence of complex numbers with $\sum\limits^{\infty}_{n=1}|a_{n}|^{2}<\infty$. Assume 
 $F=\sum\limits_{n=1}^{\infty} a_{n}f^{n}\in \operatorname{VMOA}(\mathbb{D})$. Then $f$ is a finite Blaschke product.
\begin{proof}
Assume $f$ is not a finite Blaschke product. then there exists a sequence of points $\{z_k \}$ in $\mathbb{D}$ with $|z_k| \to 1$ such that $f(z_k) \to 0$, as $k \to \infty$. Lemma 3.2 gives that 
$$
\int_{\partial\mathbb{D}} \left| F(\xi)- F(z_k) \right|^{2} P(z_k,\xi)\,dm(\xi)\ge \frac{1}{2}\, \frac{1-|f'(0)|}{1+|f'(0)|} \sum^{\infty}_{n=1}|a_{n}|^{2},
$$
if $k$ is sufficiently large. Since the integral above tends to $0$ as $k$ tends to $\infty$, we deduce that $a_n =0$ for any $n$. This finishes the proof. 
\end{proof}
\end{corollary}

\section{Open problems}\label{sec5}
We close the paper mentioning some open problems we have not explored.

% \paragraph{\ Exponential integrability.}

% Let $H$ be the sum of a lacunary series whose coefficients have square summable modulus. Then there exists $\mu_{0}>0$ such that
% $$
% \int_{\partial\mathbb{D}} e^{\mu |H|^{2}}\,dm<\infty
% $$
% for any $\mu<	\mu_{0}$. See~\cite[p.~215]{Zy}. In our context, it is natural to ask if the following % statement holds. 
% \paragraph{Problem 1.} Let $f$ be an inner function with $f(0)=0$ which is not a rotation. Let % $\{a_{n}\}$ be a sequence of complex numbers such that
% \begin{equation}\label{eq4.5}
% \sum^{\infty}_{n=1}|a_{n}|^{2}\le 1.
% \end{equation}
% Let $F=\sum\limits_{n=1}^{\infty}a_{n}f^{n}$. Does there exist a constant~$\mu>0$ such that
% $$
% \int_{\partial\mathbb{D}} e^{\mu|F|^{2}}\,dm<\infty?
% $$
% By Corollary~\ref{coro1.6} and the John--Nirenberg Theorem, condition~\eqref{eq4.5} implies that
% $$
% \int_{\partial \mathbb{D}} e^{\mu|F|}\,dm<\infty
% $$
% for any $\mu>0$. It is also natural to ask if condition~\eqref{eq4.5} implies that % $|F|^{2}\in\operatorname{BMO} (\partial\mathbb{D})$. This does not seem to be the case but we do not % have any concrete example.
\paragraph{\ Analytic continuation and smoothness.}
Hadamard's Theorem says that if a lacunary power series can be extended analytically across an arc of the unit circle, it can actually be extended analytically to a neighbourhood of the closed unit disc (see \cite[p.~208]{Zy}). It is natural to ask if the following analogous result holds.

\paragraph{Problem 1.} Let $f$ be an inner function with $f(0)=0$ which is not a rotation. Let $\{a_{n}\}$ be a sequence of complex numbers whose modulus are square summable and consider  $F=\sum\limits^{\infty}_{n=1}a_{n}f^{n}$. Assume that $F$ extends analytically across an arc~$I\subset \partial\mathbb{D}$. Is it true that there exist $n_{0}>0$ such that $a_{n}=0$ for any $n>n_{0}$ and that $f^n$ extends analytically across $I$ for any $n \leq n_0$?

\vspace{0.2cm}

For $0<\alpha<1$, let $\operatorname{Lip}_{\alpha}(\overline{\mathbb{D}})$ denote the space of continuous functions $g$ in $\overline{\mathbb{D}}$ for which there exists a constant $C=C(g)>0$ such that $|g(z) - g(w)| \leq C |z-w|^{\alpha}$, for any pair of points $z, w \in \overline{\mathbb{D}}$.  Lacunary series which belong to $\operatorname{Lip}_{\alpha}(\overline{\mathbb{D}})$ can be described in terms of their coefficients. It is then natural to ask for the corresponding result in our context. 

\paragraph{Problem 2.} Let $f$ be a finite Blaschke product with $f(0)=0$ which is not a rotation. Describe the sequences of complex numbers $\{a_{n}\}$ with square summable modulus, such that  $F=\sum\limits^{\infty}_{n=1}a_{n}f^{n}\in\operatorname{Lip}_{\alpha}(\overline{\mathbb{D}})$.

\paragraph{\ Peano curves.}
Let $\sum a_{k}z^{n_{k}}$ be a lacunary series with $\sum |a_{k}|=\infty$ and $\lim a_{k}=0$. Then for any $w\in\mathbb{C}$, there exists $\xi\in\partial\mathbb{D}$ such that $\sum a_{k}\xi^{n_{k}}$ converges to~$w$ (\cite{We}). It is then natural to ask for the following analogous result.

\paragraph{Problem 3.} Let $f$ be an inner function with $f(0)=0$ which is not a rotation. Let $\{a_{n}\}$ be a sequence of complex numbers with $\sum\limits_{n=1}^\infty |a_{n}|=\infty$, but $\lim_{n \to \infty} a_{n}=0$. Is it true that for any $w\in \mathbb{C}$, there exists $\xi\in \partial\mathbb{D}$ such that $\sum\limits_{n=1}^\infty a_{n}f^{n}(\xi)$ converges to~$w$?

% Note that if $\sum |a_{n}|^{2}=\infty$, then $F=\sum a_{n}f^{n}$ is a function in the Bloch space which has non-tangential limit at almost no point of the unit circle. In this situation, Makarov proved that for any $w\in \mathbb{C}$ there is a set $E = E(w) \subset \partial \mathbb{D}$ of Hausdorff dimension $1$, such that $F$ has radial limit $w$ at any point of the set $E$. See \cite{Ma}. 

There is also a version of the question above in the interior of the unit disc. Murai proved that any lacunary series $\sum a_{k}z^{n_{k}}$ convergent in $\mathbb{D}$ with $\sum |a_{k}|=\infty$, takes any complex value infinitely often. See \cite{Mu}. It is natural to ask for the following analogous result.

\paragraph{Problem 4.} Let $f$ be an inner function with $f(0)=0$ which is not a rotation. Let $\{a_{n}\}$ be any sequence of complex numbers such that $F(z)=\sum\limits_{n=1}^{\infty}a_{n} f^{n}(z)$ is analytic in~$\mathbb{D}$.
Assume $\sum\limits_{n=1}^{\infty}  |a_{n}|=\infty$. Is true that for any $w\in\mathbb{C}$ there exists  $z\in\mathbb{D}$ such that $F(z)=w$?

\paragraph{\ Abel's Theorem.}
Let $F(z) = \sum a_n z^n$ be a power series with radius of convergence $1$. Let $\xi \in \partial \mathbb{D}$ such that $ \sum a_n {\xi}^n$ converges. The classical Abel's Theorem says that $F(z)$ tends to $ \sum a_n {\xi}^n$, as $z$ approaches non-tangentially to $\xi$. The classical Hardy-Littlewood High Indices Theorem asserts that for lacunary series, the converse holds. In other words, if $F(z) = \sum a_k z^{n_k}$ is a lacunary series which has limit $L$ when $z$ approaches non-tangentially a point   $\xi \in \partial \mathbb{D}$, then  $ \sum a_k {\xi}^{n_k}$ converges to $L$. It is then natural to ask for the following analogous results. 

\paragraph{Problem 5.} Let $f$ be an inner function with $f(0)=0$ which is not a rotation. Let $\{a_{n}\}$ be any sequence of complex numbers such that $F(z)=\sum\limits_{n=1}^{\infty}a_{n} f^{n}(z)$ is analytic in~$\mathbb{D}$. Let  $\xi \in \partial \mathbb{D}$ such that $f^n (\xi)= \lim\limits_{r \to 1} f^n (r \xi) $ exists for any positive integer $n$. Assume  $ \sum\limits_{n=1}^\infty a_n f^n (\xi)$ converges. Is it true that 
$\lim\limits_{r \to 1} F(r \xi)$ exists? 

\paragraph{Problem 6.} Let $f$ be an inner function with $f(0)=0$ which is not a rotation. Let $\{a_{n}\}$ be any sequence of complex numbers such that $F(z)=\sum\limits_{n=1}^{\infty}a_{n} f^{n}(z)$ is analytic in~$\mathbb{D}$. Assume $F(z)$ has limit when $z$ approaches non-tangentially a point $\xi \in \partial \mathbb{D}$. Is it true that $ \sum\limits_{n=1}^\infty a_n f^n (\xi)$ converges?


\begin{thebibliography}{KWW}

\bibitem[AAN]{AAN}
\textsc{Aleksandrov, A.\ B.; Anderson, J.\ M.; Nicolau, A.}
Inner functions, Bloch spaces and symmetric measures. 
\emph{Proc.\ London Math.\ Soc.~(3)} \textbf{79} (1999), no.~2, 318--352.

\bibitem[ACP]{ACP}
\textsc{Anderson, J.\ M.; Clunie, J.; Pommerenke, Ch.}
On Bloch functions and normal functions.
\emph{J.\ Reine Angew.\ Math.} \textbf{270} (1974), 12--37.

\bibitem[Ba]{Ba}
\textsc{Baernstein, A.} Analytic functions of bounded mean oscillation
\emph{Aspects of contemporary complex analysis }. 
Proc. NATO Adv. Study Inst., Univ. Durham, 1979, pp. 3–36. 
 Academic Press, London-New York, 1980.

\bibitem[BG]{BG}
\textsc{B\"ottcher, A.; Grudsky, S.\ M.}
\emph{Toeplitz matrices, asymptotic linear algebra and functional analysis}. 
Birk\"auser Verlag, Basel, 2000.

\bibitem[EG]{EG}
\textsc{Evans, L.\ C.; Gariepy, R.\ F.}
\emph{Measure theory and fine properties of functions}.  Revised edition.
Textbooks in Mathematics. CRC Press, Boca Raton, FL, 2015.

% \bibitem[FMP1]{FMP1}
% \textsc{Fern\'andez, J.\ L.; Meli\'an M.\ V.; Pestana, D.}
% Quantitative mixing results and inner functions.
% \emph{Math.\ Ann.} \textbf{337} (2007), no.~1, % 233--251.

% \bibitem[FMP2]{FMP2}
% \textsc{Fern\'andez, J.\ L.; Meli\'an M.\ V.; Pestana, D.}
% Expanding maps, shrinking targets and hitting times.
% \emph{Nonlinearity} \textbf{25} (2012), no.~9, 2443--2471.

\bibitem[Ga]{Ga}
\textsc{Garnett, J.\ B.}
\emph{Bounded analytic functions}. Revised first edition. Graduate Texts in Mathematics~\textbf{236}.
Springer, New York, 2007.


\bibitem[GJ]{GJ}
\textsc{Garnett, J.\ B.; Jones, P.\ W.}
The distance in BMO to $L^\infty$
\emph{Ann. of Math.} \textbf{108} (1978), no.~2, 373--393.
 
% \bibitem[GZ]{GZ}
% \textsc{Ghatage, P.\ G.; Zheng, C.}
% Analytic functions of bounded mean oscillation and the Bloch space.
% \emph{Integral Equations Operator Theory} \textbf{17} (1993), no.~4, 501--515.

\bibitem[Ka]{Ka}
\textsc{Kahane, J.-P.}
\emph{Some random series of functions}. Second edition. Cambridge Studies in Advanced Mathematics~\textbf{5}.
Cambridge University Press, Cambridge, 1985.

% \bibitem[KWW]{KWW}
% \textsc{Kahane, J.-P.; Weiss, M.; Weiss, G.}
% On lacunary power series.
% \emph{Ark.\ Mat.} \textbf{5} (1963), 1--26.

% \bibitem[Ma]{Ma}
% \textsc{Makarov, N. G.}
% Probability methods in the theory of conformal mappings. 
% \emph{(Russian) Algebra i Analiz} \textbf{1} (1989), no.~1., 3--59; translation in Leningrad Math. J. 1 (1990), no. 1, 1–56 

% \bibitem[Mt]{Mt}
% \textsc{Mattila, P.}
% \emph{Geometry of sets and measures in Euclidean spaces. Fractals and rectifiability}.  Cambridge % Studies in Advanced Mathematics~\textbf{44}.
% Cambridge University Press, Cambridge, 1995.

% \bibitem[MN]{MN}
% \textsc{Monreal, N.; Nicolau, A.}
% The closure of the Hardy space in the Bloch norm. 
% \emph{Algebra i Analiz} \textbf{22} (2010), no.~1., 75--81; translation in St. Petersburg Math. J. \textbf{22} (2011), no. 1, 55-59

\bibitem[Mu]{Mu}
\textsc{Murai, T.}
The value-distribution of lacunary series and a conjecture of Paley.
\emph{Ann. Inst. Fourier} \textbf{31} (1981), no.~1. vii, 135--156.

\bibitem[Ne]{Ne}
\textsc{Neuwirth, J.\ H.}
Ergodicity of some mappings of the circle and the line.
\emph{Israel J.\ Math.} \textbf{31} (1978), no.~3.--4, 359--367.


\bibitem[NS1]{NS1}
\textsc{Nicolau, A.; Soler i Gibert, O.}
Approximation in the Zygmund class.
\emph{J.\ London \ Math.\ Soc.~(2)} \textbf{101} (2020), no.~1, 226--246.

\bibitem[NS2]{NS2}
\textsc{Nicolau, A.; Soler i Gibert, O.}
A central limit theorem for inner functions.
Preprint, 2020. arXiv:2006.12105


% \bibitem[PZ1]{PZ1}
% \textsc{Paley, R.\ E.\ A.\ C.; Zygmund, A.}
% (1), (2), (3): On some series of functions.
% (1),~(2):
% \emph{Proc.\ Cambridge Philos.\ Soc.} \textbf{26} (1930), 337--357, 458--474, (3): \emph{Proc.\ Cambridge Philos.\ Soc.} \textbf{28} (2) (1932), 266--272. 

% \bibitem[PZ2]{PZ2}
% \textsc{Paley, R.\ E.\ A.\ C.; Zygmund, A.}
% A note on analytic functions in the unit\linebreak circle.
% \emph{Math.\ Proc.\ Cambridge Philos.\ Soc.} \textbf{28} (1932), no.~3, 266--272.

\bibitem[Po]{Po}
\textsc{Pommerenke, Ch.}
On ergodic properties of inner functions.
\emph{Math.\ Ann.} \textbf{256} (1981), no.~1, 43--50.

\bibitem[SS]{SS}
\textsc{Saksman, E.; Soler i Gibert, O.}
Approximation in the Zygmund and H\"older classes on $\mathbb{R}^{n}$.
\href{https://arxiv.org/abs/2009.09752}
{arXiv:2009.09752}.


\bibitem[Sh]{Sh}
\textsc{Shapiro, J.\ H.}
\emph{Composition operators and classical function theory}.  Universitext: Tracts in Mathematics.
Springer-Verlag, New York, 1993.


\bibitem[Sm]{Sm}
\textsc{Smith, W.}
Inner functions in the hyperbolic little Bloch class.
\emph{Michigan Math.\ J.} \textbf{45} (1998), no.~1, 103--114.

\bibitem[We]{We}
\textsc{Weiss, M.}
Concerning a theorem of Paley on lacunary power series.
\emph{Acta Math.} \textbf{102} (1959), 225--238.

\bibitem[Zy]{Zy}
\textsc{Zygmund, A.}
\emph{Trigonometric series}. Vol.~I, II. Third edition. With a foreword by Robert A.~Fefferman. Cambridge Mathematical Library. Cambridge University Press, Cambridge, 2002.

\end{thebibliography}
\end{document}